\documentclass[12pt]{amsart}

\usepackage[utf8]{inputenc}						
\usepackage{amsmath,amsfonts,amssymb,amsthm}	

\usepackage{geometry}		
\usepackage{mathtools}		
\usepackage{hyperref}		
\usepackage{setspace}		
\usepackage{cite}			
\usepackage{color}			

\definecolor{midpurple}{rgb}{0.6,0.2,0.4}
\hypersetup{colorlinks=true,citecolor=blue,linkcolor=blue,urlcolor=midpurple}

\numberwithin{equation}{section}			
\newtheorem{theorem}{Theorem}[section]			
\newtheorem{proposition}[theorem]{Proposition}	

\newtheorem{definition}[theorem]{Definition}
\newtheorem{corollary}[theorem]{Corollary}

\newtheorem{fact}[theorem]{Fact}

\newtheorem*{theorem*}{Theorem}
\newtheorem*{proposition*}{Proposition}
\newtheorem*{lemma*}{Lemma}
\newtheorem*{corollary*}{Corollary}
\newtheorem*{remark*}{Remark}
\newtheorem*{definition*}{Definition}

\renewcommand{\leq}{\leqslant}	
\renewcommand{\geq}{\geqslant}	

\newcommand{\N}{\mathbb{N}}		
\newcommand{\Z}{\mathbb{Z}}		
\newcommand{\R}{\mathbb{R}}		
\newcommand{\C}{\mathbb{C}}		
\newcommand{\T}{\mathbb{T}}		

\newcommand{\eps}{\varepsilon}			
\newcommand{\cjg}{\overline}			

\newcommand{\wt}{\widetilde}		
\newcommand{\wh}{\widehat}			
\DeclareMathOperator{\Supp}{Supp}	

\DeclareMathOperator{\rk}{rk}			

\newcommand{\transp}{\mathsf{T}}		

\newcommand{\Nzero}{\mathbb{N}_{0}}		

\newcommand{\cont}[1]{\mathcal{C}(\mathbb{R}^{#1})}						
\newcommand{\schw}[1]{\mathcal{S}(\mathbb{R}^{#1})}						
\newcommand{\bump}[1]{\mathcal{C}^\infty_c(\mathbb{R}^{#1})}			
\newcommand{\cutoff}[1]{\mathcal{C}^\infty_{c,+}(\mathbb{R}^{#1})}		
\newcommand{\meas}[1]{\mathcal{M}^+(\mathbb{R}^{#1})}					

\newcommand{\dx}{\mathrm{d}x}				
\newcommand{\dy}{\mathrm{d}y}				
\newcommand{\du}{\mathrm{d}u}				
\newcommand{\dv}{\mathrm{d}v}				
\newcommand{\dxi}{\mathrm{d}\xi}			
\newcommand{\dkappa}{\mathrm{d}\kappa}		
\newcommand{\dtheta}{\mathrm{d}\theta}		
\newcommand{\dlambda}{\mathrm{d}\lambda}	
\newcommand{\deta}{\mathrm{d}\eta}			
\newcommand{\dsigma}{\mathrm{d}\sigma}		
\newcommand{\dmu}{\mathrm{d}\mu}			

\newcommand{\dbfxi}{\mathrm{d}\boldsymbol{\xi}}		
\newcommand{\dbfeta}{\mathrm{d}\boldsymbol{\eta}}	

\newcommand{\calH}{\mathcal{H}}				
\newcommand{\calF}{\mathcal{F}}				
\newcommand{\calM}{\mathcal{M}}				

\newcommand{\bfA}{\mathbf{A}}				
\newcommand{\bfB}{\mathbf{B}}				
\newcommand{\bfR}{\mathbf{R}}				
\newcommand{\bfxi}{\boldsymbol{\xi}}		
\newcommand{\bfeta}{\boldsymbol{\eta}}		

\newcommand{\leb}{\mathcal{L}}				
\newcommand{\dleb}{\mathrm{d}\mathcal{L}}	
\newcommand{\cube}[1]{\big[\! -\tfrac{1}{#1},\tfrac{1}{#1} \big]^n}

\begin{document}

\title{On polynomial configurations in fractal sets}

\author{K. Henriot, I. \L{}aba, M. Pramanik}

\date{}

\begin{abstract}
	We show that subsets of $\R^n$
	of large enough Hausdorff and Fourier dimension
	contain polynomial patterns of the form
	\begin{align*}
		( x ,\, x + A_1 y ,\, \dots,\, x + A_{k-1} y ,\, x + A_k y + Q(y) e_n ),
		\quad x \in \R^n,\ y \in \R^m, 
	\end{align*}
	where $A_i$ are real $n \times m$ matrices,
	$Q$ is a real polynomial in $m$ variables
	and $e_n = (0,\dots,0,1)$.
\end{abstract}

\maketitle

\section{Introduction}
\label{sec:intro}

In this work we investigate the presence
of point configurations in
subsets of $\R^n$ which are large in a certain sense. 
When $E$ is a subset of $\R^n$ of positive
Lebesgue measure, a consequence of the
Lebesgue density theorem is that
$E$ contains a similar copy of any finite set.
A more difficult result of Bourgain~\cite{Bourgain:DilatesConfig}
states that sets of positive upper density in $\R^n$ contain,
up to isometry, all large enough dilates of 
the set of vertices of 
any fixed non-degenerate $(n-1)$-dimensional simplex.
In a different setting, Roth's theorem~\cite{Roth:RothI}
in additive combinatorics states that subsets of 
$\Z$ of positive upper density
contain non-trivial three-term arithmetic progressions.

When a subset $E \subset \R$ is only supposed to have
a positive Hausdorff dimension,
a direct analogue of Roth's theorem is impossible.
Indeed Keleti~\cite{Keleti:NoConfigs} has constructed a set of full dimension in $[0,1]$
not containing the vertices of any non-degenerate parallelogram,
and in particular not containing any non-trivial three-term arithmetic progression.
Maga~\cite{Maga:NoConfigs} has since extended this construction to dimensions $n \geq 2$.
The work of \L{}aba and Pramanik~\cite{LP:Configs} 
and its multidimensional extension by Chan et al.~\cite{CLP:Configs}
circumvent these obstructions
under additional assumptions on the set $E$, 
which we now describe.

When $E$ is a compact subset of $\R^n$,
Frostman's lemma~\cite[Chapter~8]{Wolff:Book} essentially states that
its Hausdorff dimension is equal to
\begin{align*}
	\dim_{\calH} E &= 
	\sup\{ \alpha \in [0,n) \,:\, \exists\ \mu \in \calM(E) \,:\, 
	\sup\limits_{x \in \R^n,\, r >0} \mu\big[ B(x,r) \big] r^{-\alpha} < \infty \},
\end{align*}
where $\calM(E)$ is the space
of probability measures supported on $E$.
On the other hand, the Fourier dimension of $E$ is
\begin{align*}
	\dim_{\calF} E &= \sup\{ \beta \in [0,n) \,:\, \exists\ \mu \in \calM(E) \,:\, 
					\sup\limits_{\xi \in \R^n} |\wh{\mu}(\xi)| (1+|\xi|)^{-\beta/2}  < \infty \}.
\end{align*}
It is well-known that we have $\dim_{\calF}(E) \leq \dim_{\calH}(E)$ 
for every compact set $E$,
with strict inequality in many instances, 
and we call $E$ a Salem set when equality holds.
There are various known constructions of Salem 
sets~\cite{Salem:Salem,Kaufman:Salem,Bluhm:SalemI,Bluhm:SalemII,Kahane:Book,LP:Configs,Hambrooke:Salem},
several of which~\cite{Korner:Salem,Chen:Salem} also produce sets 
with prescribed Hausdorff and Fourier dimensions $0 < \beta \leq \alpha < n$.

In a very abstract setting,
one may ask whether it is possible to find 
translation-invariant patterns of the form
\begin{align}
\label{eq:intro:ShiftsPattern}
	\Phi(x,y) = ( x , x + \varphi_1(y) , \dots, x + \varphi_k (y) ) 
\end{align}
in the product set $E \times \dots \times E$, 
where the $\varphi_j : \Omega \subset \R^m \rightarrow \R^n$
are certain shift functions.
When $n + m \geq (k+1)n$,
the map $\Phi$ considered is often a 
submersion of an open subset of $\R^{n + m}$ onto $\R^{(k+1)n}$,
and then one can find a pattern of the desired kind in $E$
via the implicit function theorem.
A natural restriction is therefore to assume that $m < kn$
in this multidimensional setting.
Chan et al.~\cite{CLP:Configs}
studied the case where the maps $\varphi_j(y) = A_j y$ 
are linear for matrices $A_j \in \R^{n \times m}$, 
generalizing the study of \L{}aba and Pramanik
for three-term arithmetic progressions,
under the following technical assumption.

\begin{definition}
\label{thm:intro:NonDegen}
Let $n,k,m \geq 1$ and suppose that 
$m = (k - r)n + n'$ with $1 \leq r < k$ and $0 \leq n' < n$.
We say that the system of matrices
$A_1,\dots,A_k \in \R^{n \times m}$ is non-degenerate
when 
\begin{align*}
	\rk \begin{bmatrix}
		A_{j_1}^\transp & \hdots & A_{j_{k - r + 1}}^\transp	\\
		I_{n \times n} & \hdots & I_{n \times n} 
	\end{bmatrix} = (k - r + 1)n,
\end{align*}
for every set of indices $\{j_1,\dots,j_{k - r + 1}\} \subset \{ 0,\dots,k \}$,
with the convention that $A_0 = 0_{n \times n}$.
\end{definition}

Requirements similar to the above arise when analysing linear patterns 
by ordinary Fourier analysis in additive combinatorics~\cite{Roth:RothII},
and there is a close link with the modern definition of linear systems 
of complexity one~\cite{GW:Complexity}.
The main result of Chan et al.~\cite{CLP:Configs}
gives a fractal analogue of the multidimensional Szemerédi theorem~\cite{FK:MultiSzemeredi}
for non-degenerate linear systems,
when the Frostman measure has both dimensional and Fourier decay.
We only state it in the case where $n$ divides $m$ for simplicity.

\begin{theorem}[Chan, \L{}aba and Pramanik]
\label{thm:intro:CLP}
Let $n,k,m \geq 1$,
$D \geq 1$ and $\alpha, \beta \in (0,n)$.
Suppose that $E$ is a compact subset of $\R^n$
and $\mu$ is a probability measure supported on $E$
such that\footnote{
In fact, this theorem was proved in~\cite{CLP:Configs} under the more
restrictive condition $|\wh{\mu}(\xi)| \leq D(1+|\xi|)^{-\beta/2}$
for a fixed constant $D$.
However, by examining the proof there,
one can see that the constant 
$D = D_\alpha$ may be allowed to grow polynomially in $n-\alpha$,
as was the case in the original argument
of \L{}aba and Pramanik~\cite{LP:Configs}.
}
\begin{align*}
	\mu\big[ B(x,r) \big]
	\leq Dr^\alpha
	\quad\text{and}\quad
	|\wh{\mu}(\xi)| \leq D(n-\alpha)^{-D} (1+|\xi|)^{-\beta/2}
\end{align*}
for all $x \in \R^n$, $r >0$ and $\xi \in \R^n$. 
Suppose that $(A_1,\dots,A_k)$ is a non-degenerate system
of $n \times m$ matrices
in the sense of Definition~\ref{thm:intro:NonDegen}.
Assume finally that
$m = (k-r) n$ with $1 \leq r < k$
and, for some $\eps \in (0,1)$,
\begin{align*}
	\Big\lceil \frac{k}{2} \Big\rceil  n \leq m < kn,
	\qquad
	\frac{2(kn - m)}{k+1} + \eps \leq \beta < n,
	\qquad
	n - c_{n,k,m,\eps,D,(A_i)} \leq \alpha < n,
\end{align*}
for a sufficiently small constant $c_{n,k,m,\eps,D,(A_i)} > 0$.
Then, for every collection of strict subspaces $V_1,\dots,V_q$ of $\R^{n+m}$,
there exists $(x,y) \in \R^{n + m} \smallsetminus V_1 \cup \dots V_q$
such that
\begin{align*}
	( x , x + A_1 y , \dots, x + A_k y ) \in E^{k+1}.
\end{align*}
\end{theorem}

Note that the Hausdorff dimension $\alpha$ is required to be large enough
with respect to the constants involved in the dimensional
and Fourier decay bounds for the Frostman measure. 
A construction due to Shmerkin~\cite{Shmerkin:Configs} shows that 
the dependence of $\alpha$ on the constants cannot be removed.

In practice, Salem set constructions provide
a family of fractal sets indexed by $\alpha$,
and it is often possible to verify the conditions of
Theorem~\ref{thm:intro:CLP} for $\alpha$ close to $n$ ; 
this was done in a number of cases in~\cite{LP:Configs}.
The requirement of Fourier decay
of the measure $\mu$ serves as an analogue
of the notion of pseudorandomness in additive combinatorics~\cite{TV:Book},
under which we expect a set to contain the same
density of patterns as a random set of same size.

In this work we consider a class of polynomial patterns,
which generalizes that of Theorem~\ref{thm:intro:CLP}.
We aim to obtain results similar in spirit to the Furstenberg-Sarközy
theorem~\cite{Sarkozy:Sarkozy,Furstenberg:Sarkozy} in additive combinatorics, 
which finds patterns of the form $( x , x + y^2)$ in dense subsets of $\Z$.
A deep generalization of this result is
the multidimensional polynomial Szemerédi theorem
in ergodic theory~\cite{BL:PolSzemeredi,BM:PolSzemeredi}
(see also~\cite[Section~6.3]{BLL:PolSzemeredi}),
which handles patterns of the form~\eqref{eq:intro:ShiftsPattern}
where each shift function $\varphi_j$
is an integer polynomial vector with zero constant term.
By contrast, the class of patterns we study includes only one polynomial term,
which should satisfy certain non-degeneracy conditions.
We are also forced to work with a dimension $n \geq 2$, 
and all these limitations are due to the inherent
difficulty in analyzing polynomial patterns through Fourier analysis.
On the other hand, we are able to relax
the Fourier decay condition on the fractal measure 
needed in Theorem~\ref{thm:intro:CLP}.

\begin{theorem}
\label{thm:intro:MainThm}
Let $n,m,k \geq 2$, $D \geq 1$ and $\alpha, \beta \in (0,n)$.
Suppose that $E$ is a compact subset of $\R^n$
and $\mu$ is a probability measure supported on $E$
such that
\begin{align*}
	\mu\big[ B(x,r) \big]
	\leq Dr^\alpha
	\quad\text{and}\quad
	|\wh{\mu}(\xi)| \leq D(n-\alpha)^{-D} (1+|\xi|)^{-\beta/2}
\end{align*}
for all $x \in \R^n$, $r >0$ and $\xi \in \R^n$. 
Suppose that $(A_1,\dots,A_k)$ is a non-degenerate system of real $n \times m$ matrices
in the sense of Definition~\ref{thm:intro:NonDegen}.
Let $Q$ be a real polynomial in $m$ variables such that $Q(0) = 0$ 
and the Hessian of $Q$ does not vanish at zero.
Assume furthermore that, for a constant $\beta_0 \in (0,n)$,
\begin{align*}
	(k-1) n  < m < kn,
	\qquad
	\beta_0 \leq \beta < n,
	\qquad
	n - c_{\beta_0,n,k,m,D,(A_i),Q} < \alpha < n,
\end{align*}
for a sufficiently small constant $c_{\beta_0,n,k,m,D,(A_i),Q} > 0$.
Then, for every collection $V_1,\dots,V_q$ of 
strict subspaces of $\R^{m+n}$,
there exists $(x,y) \in \R^{n + m} \smallsetminus (V_1 \cup \cdots \cup V_q)$
such that
\begin{align}
\label{eq:intro:OnePolPattern}
	(x,\, x + A_1 y, \dots,\, x + A_{k-1} y ,\, x + A_k y + Q(y) e_n) \in E^{k+1},
\end{align}
where $e_n = (0,\dots,0,1)$.
\end{theorem}

Our argument follows broadly the transference strategy
devised by \L{}aba and Pramanik~\cite{LP:Configs},
and its extension by Chan and these two authors~\cite{CLP:Configs}.
However, the case of polynomial configurations requires
a more delicate treatment of the singular integrals arising in the analysis.
The weaker condition on $\beta$ is obtained
by exploiting restriction estimates for fractal measures
due to Mitsis~\cite{Mitsis:RestrFrac} 
and Mockenhaupt~\cite{Mockenhaupt:RestrFrac}.
A more detailed outline of our strategy can
be found in Section~\ref{sec:scheme}.
By the method of this paper, one can also obtain an analogue
of Theorem~\ref{thm:intro:CLP} with the same relaxed condition on 
the exponent $\beta$, and we state this version
precisely in Section~\ref{sec:lin}.

For concreteness' sake, we highlight the lowest dimensional situation
handled by Theorem~\ref{thm:intro:MainThm}.
When $k=n=2$ and $m=3$, this theorem allows us to detect patterns
of the form
\begin{align*}
	\begin{bmatrix} x_1 \\ x_ 2 \end{bmatrix},\
	\begin{bmatrix} x_1 \\ x_ 2 \end{bmatrix}
	+ A_1 \begin{bmatrix} y_1 \\ y_2 \\ y_3 \end{bmatrix},
	\begin{bmatrix} x_1 \\ x_ 2 \end{bmatrix}
	+ A_2 \begin{bmatrix} y_1 \\ y_2 \\ y_3 \end{bmatrix}
	+ \begin{bmatrix} 0 \\ Q(y_1,y_2,y_3) \end{bmatrix},
\end{align*}
for matrices $A_1,A_2 \in \R^{2 \times 3}$ of full rank
such that $A_1 - A_2$ has full rank,
and for a non-degenerate quadratic form $Q$ in three variables.
We may additionally impose that 
$y_1,y_2,y_3 \in \R \smallsetminus \{0\}$ 
by setting $V_i = \{ (x,y) \in \R^5 \,:\, y_i = 0\}$
in Theorem~\ref{thm:intro:MainThm}.
For example, when
$A_1 = \bigl[ \begin{smallmatrix} 1 & 0 & 0 \\ 0 & 1 & 0 \end{smallmatrix} \bigr]$,
$A_2 = \bigl[ \begin{smallmatrix} 0 & 0 & 1 \\ 1 & 0 & 0 \end{smallmatrix} \bigr]$
and $Q(y) = |y|^2$, we can detect the configuration
\begin{align*}
	\begin{bmatrix} x_1 \\ x_ 2 \end{bmatrix},\
	\begin{bmatrix} x_1 + y_1 \\ x_ 2 + y_2 \end{bmatrix},\,
		\begin{bmatrix} x_1 + y_3 \\ x_ 2 + y_1 + y_1^2 + y_2^2 + y_3^2 \end{bmatrix}
\end{align*}
with $y_1,y_2,y_3 \in \R \smallsetminus \{0\}$.
However, we cannot detect the configuration
\begin{align*}
	(x,x + y,x + y^2),\quad x \in \R,\, y \in \R \smallsetminus \{0\},
\end{align*}
for then we have $n=m=1$ and $k=2$, and
the condition $m > (k-1) n$ is not satisfied.

Note also that, in the statement of Theorem~\ref{thm:intro:MainThm}, 
one may add a linear term in variables $y_1,\dots,y_m$ to
the polynomial $Q$ without affecting the assumptions on it.
This allows for some flexibility in satisfying the matrix
non-degeneracy conditions of Definition~\ref{thm:intro:NonDegen},
since one may alter the last line of $A_k$ at will.
For example, the degenerate system of matrices
$\bigl[ \begin{smallmatrix} 1 & 0 & 0 \\ 0 & 1 & 0 \end{smallmatrix} \bigr]$,
$\bigl[ \begin{smallmatrix} 0 & 0 & 1 \\ 0 & 0 & 0 \end{smallmatrix} \bigr]$
and the polynomial $Q(y) = |y|^2$
give rise to the configuration
\begin{align*}
	\begin{bmatrix} x_1 \\ x_ 2 \end{bmatrix},\
	\begin{bmatrix} x_1 + y_1 \\ x_ 2 + y_2 \end{bmatrix},\,
	\begin{bmatrix} x_1 + y_3 \\ x_ 2 + |y|^2 \end{bmatrix}.
\end{align*}
Rewriting $|y|^2 = y_1 + y_2 + y_3 + Q_1(y)$,
we see that $Q_1$ still has non-degenerate Hessian at zero
and the configuration is now associated to the system of matrices
$\bigl[ \begin{smallmatrix} 1 & 0 & 0 \\ 0 & 1 & 0 \end{smallmatrix} \bigr]$,
$\bigl[ \begin{smallmatrix} 0 & 0 & 1 \\ 1 & 1 & 1 \end{smallmatrix} \bigr]$,
which is easily seen to be non-degenerate.
One possible explanation for this curious phenomenon is that,
by comparison with the setting of Theorem~\ref{thm:intro:CLP},
we have an extra variable at our disposition,
since $m > (k-1)n \geq \lceil k/2 \rceil n$.

Finally, we note that there is a large body of literature
on configurations in fractal sets where Fourier decay assumptions are not required.
Here, the focus is often on finding a large variety (in a specified quantitative sense) of certain types
of configurations.
A well-known conjecture of Falconer~\cite[Chapter~9]{Wolff:Book} 
states that when a compact subset $E$ of $\R^n$ has Hausdorff dimension 
at least $n/2$, its set of distances $\Delta(E) = \{ |x-y|,\, x,y \in E \}$
must have positive Lebesgue measure. 
This can be phrased in terms of $E$ containing configurations
$\{x,y\}$ with $|x-y|=d$ for all $d\in \Delta(E)$, where $\Delta(E)$ is ``large."
Wolff~\cite{Wolff:DistanceSets} and Erdo\~gan~\cite{Erdogan:DistanceSets1,Erdogan:DistanceSets2} 
proved that the distance set $\Delta(E)$ has positive Lebesgue measure for $\dim_{\mathcal{H}} E > \frac{n}{2} + \frac{1}{3}$,
and Mattila and Sjölin~\cite{MS:DistanceSets} showed that 
it contains an open interval for $\dim_{\mathcal{H}} E > \frac{n+1}{2}$.
More recently, Orponen~\cite{Orponen:DistanceSets} proved using very different methods that 
$\Delta(E)$ has upper box dimension 1 if $E$ is $s$-Ahlfors-David regular with $s\geq 1$. 
There is a rich literature generalizing these results
to other classes of configurations,
such as triangles~\cite{GI:GenConfigs}, simplices~\cite{GGIP:GenConfigs,GILP:GenConfigs},
or sequences of vectors with prescribed consecutive lengths~\cite{BIT:Chains,GIP:Necklaces}.

In a sense, the configurations studied in these
references enjoy a greater degree of directional freedom,
which ensures that they are not avoided by sets of full Hausdorff dimension.
By contrast, a Fourier decay assumption is necessary
to locate $3$-term progressions in a fractal set of
full Hausdoff dimension (as mentioned earlier), 
and in light of recent work of Mathé~\cite{Mathe:NoConfigs}, 
it is likely that a similar assumption 
is needed to find polynomial patterns of the form~\eqref{eq:intro:OnePolPattern}.
It is, however, possible that our non-degeneracy assumptions are not optimal, or that special cases of our results 
could be proved without Fourier decay 
assumptions\footnote{
After this article was first submitted for publication, 
a result of this type was indeed proved by Iosevich and Liu~\cite{IL:Triangles}.}. 
Loosely speaking, we would expect that configurations
with more degrees of freedom are less likely to require Fourier conditions, but the specifics are
far from understood and we do not feel that we have sufficient
data to attempt to make a conjecture in this direction.

\textbf{Acknowledgements.}
This work was supported by NSERC Discovery grants
22R80520 and 22R82900.

\section{Notation}
\label{sec:notation}

We define the following standard spaces of complex-valued functions and measures:
\begin{align*}
	\cont{d} 	&= \{ \text{continuous functions on $\R^d$} \}, \\
	\schw{d} 	&= \{ \text{Schwartz functions on $\R^d$} \}, \\
	\bump{d} 	&= \{ \text{smooth compactly-supported functions on $\R^d$} \}, \\
	\cutoff{d}	&= \{ \text{non-negative smooth compactly-supported functions on $\R^d$} \}, \\
	\meas{d}	&= \{ \text{finite non-negative Borelian measures on $\R^d$} \}.
\end{align*}
Similar notation is employed for functions on $\T^d$.
We write $e(x) = e^{2i\pi x}$ for $x \in \R$.
We let $\mathcal{L}$ denote either the Lebesgue measure on $\R^d$ or 
the normalized Haar measure on $\T^d$.
We let $\dsigma$ denote generically the Euclidean surface measure on a submanifold of $\R^d$.
When $f$ is a function on an abelian group $G$ and $t$ is an element of $G$, 
we denote the $t$-shift of $f$ by $T^t f (x) = f( x + t )$.
When $A$ is a matrix we denote its transpose by $A^\transp$.
We also write $[n] = \{1,\dots,n\}$ for an integer $n$
and $\Nzero = \N \cup \{0\}$.

\section{Broad scheme}
\label{sec:scheme}

In this section we introduce the basic objects
that we will work with in this paper.
We also state the intermediate propositions corresponding
to the main steps of our argument,
and we derive Theorem~\ref{thm:intro:MainThm} from them
at the outset.

We fix a compact set $E \subset \R^n$
and a probability measure $\mu$ supported on $E$. 
For technical reasons, 
we suppose that $E \subset \cube{16}$.
We fix two exponents $0 < \beta \leq \alpha < n$,
as well as two constants $D, D_\alpha \geq 1$,
where the subscript in the second constant 
indicates that it is allowed to vary with $\alpha$.
We assume that the measure $\mu$ verifies
the following dimensional and Fourier decay conditions:
\begin{align}
	\label{eq:scheme:BallDecay}
	&\phantom{(x \in \R^n, r > 0 )} &
	\mu\big[B(x,r)\big] &\leq D r^{\alpha} &
	&(x \in \R^n, r > 0), \\
	\label{eq:scheme:FourierDecay}
	&\phantom{(\xi \in \R^n)} &	
	|\wh{\mu}(\xi)| &\leq D_{\alpha} (1 + |\xi|)^{-\beta/2}
	&&(\xi \in \R^n).
\end{align}
We suppose that the second constant involved
blows up (if at all) at most polynomially as $\alpha$ tends to $n$:
\begin{align}
\label{eq:scheme:Dgrowth}
	D_{\alpha} \lesssim (n - \alpha)^{-O(1)}.
\end{align}

We also let $k \geq 3$ and we consider smooth functions
$\varphi_1,\dots,\varphi_k : \Omega \subset \R^m \rightarrow \R^n$,
where $\Omega$ is an open neighborhood of zero.
We are interested in locating the pattern
\begin{align}
\label{eq:scheme:ShiftsPattern}
	\Phi(x,y) = (x, x + \varphi_1(y), \dots, x + \varphi_k(y) )
\end{align}
in $E^{k+1}$.
While this abstract notation is sometimes useful,
in practice we work with the maps
\begin{align}
\label{eq:scheme:OnePolPattern}
	(\varphi_1(y),\dots,\varphi_k(y)) 
	= ( A_1 y, \dots ,\, A_{k-1} y ,\, A_k y + Q(y)e_n ),
\end{align}
where $(A_1,\dots,A_k)$
is a non-degenerate system of $n \times m$ matrices
in the sense of Definition~\ref{thm:intro:NonDegen}
and $Q \in \R[y_1,\dots,y_m]$
is such that $Q(0) = 0$ and the Hessian of $Q$ does not vanish at zero.
We also fix a smooth cutoff $\psi \in \cutoff{m}$ supported on $\Omega$
such that $\psi \geq 1$ on a small box $[-c,c]^m$
and the Hessian of $Q$ is bounded away from zero on the support of $\psi$.
This cutoff is used in Definition~\ref{thm:scheme:ConfigOp} below.
We take the opportunity here to state an equivalent form
of Definition~\ref{thm:intro:NonDegen} when $m \geq (k-1)n$.

\begin{definition}
\label{thm:scheme:NonDegen}
If $m \geq (k-1)n$, we say that
the system of matrices $(A_i)_{1 \leq i \leq k}$ with $A_i \in \R^{n \times m}$
is non-degenerate when, 
for every $1 \leq j \leq k$
and writing $[k] = \{ i_1,\dots,i_{k-1},j \}$,
the matrices 
\begin{align*}
	[ A_1^\transp \ \hdots\  \wh{A}_j^\transp \ \hdots \ A_k^\transp ],
	\qquad
	[ (A_{i_1}^\transp - A_j^\transp) \ \hdots \ (A_{i_{k-1}}^\transp - A_j^\transp)]
\end{align*}
(where the hat indicates omission) have rank $(k-1)n$.
\end{definition}

We also state a few notational conventions 
applied throughout the article.
When $(A_1,\dots,A_k)$ is a system of $n \times m$ matrices,
we define the $kn \times m$ matrix $\bfA$ by
$\bfA^\transp = [ A_1^\transp \,\dots\, A_k^\transp ]$.
Unless mentioned otherwise,
we allow every implicit or explicit constant in the article
to depend on the integers $n,k,m$, the constant $D$,
the matrices $A_i$ and the polynomial $Q$,
and the cutoff function $\psi$.
This convention is already in effect in the
propositions stated later in this section.

We start by defining a multilinear form 
which plays a central role in our argument. 

\begin{definition}[Configuration form]
\label{thm:scheme:ConfigOp}
For functions $f_0, \dots, f_k \in \schw{n}$, we let
\begin{align*}
	\Lambda(f_0,\dots,f_k) 
	= \int_{\R^n} \int_{\R^m} f_0(x) f_1(x + \varphi_1(y)) \cdots f_k(x + \varphi_k(y)) \dx\, \psi(y) \dy.
\end{align*}
\end{definition}

In Section~\ref{sec:ops}, we show that
the multilinear form has the following
convenient Fourier expression.

\begin{proposition}
\label{thm:scheme:FourierInv}
For measurable functions $F_0,\dots,F_k$ on $\R^n$ and $K$ on $\R^{nk}$, we let
\begin{align}
\label{eq:scheme:SingIntg}
	\Lambda^*( F_0, \dots, F_k ; K )
	= \int_{(\R^n)^k} F_0( - \xi_1 - \dotsb - \xi_k )
	F_1( \xi_1 ) \cdots F_k( \xi_k ) K(\bfxi) \dbfxi,
\end{align}
whenever the integral is absolutely convergent
or the integrand is non-negative.
For every $f_0,\dots,f_k \in \schw{n}$, we have
\begin{align*}
	\Lambda( f_0, \dots, f_k )
	= \Lambda^*( \wh{f}_0, \dots, \wh{f}_k ; J),
\end{align*}
where $J$ is the oscillatory integral of Definition~\ref{thm:ops:OscIntg}.
\end{proposition}

We may extend the configuration operator to measures,
whenever we have absolute convergence of the dual form.

\begin{definition}[Configuration form for measures]
\label{thm:scheme:OpMeasures}
When $\lambda_0,\dots,\lambda_k \in \meas{n}$
are such that $\Lambda^*(\, |\wh{\lambda}_0|,\dots,|\wh{\lambda}_k| ; |J| \,) < \infty$, 
we define
\begin{align*}
	\Lambda(\lambda_0,\dots,\lambda_k) 
	= \Lambda^*(\wh{\lambda}_0,\dots,\wh{\lambda}_k ; J).
\end{align*}
When $\lambda_j \in \schw{n}$,
this is compatible with Definition~\ref{thm:scheme:ConfigOp}
by Proposition~\ref{thm:scheme:FourierInv}.
\end{definition}

The next step, carried out in Section~\ref{sec:intg},
is to obtain bounds for the dual multilinear form
evaluated at the Fourier-Stieljes transform of the fractal measure $\mu$.
Such bounds hold only in certain ranges of $\alpha,\beta$ 
and under certain restrictions on $n,k,m$.

\begin{proposition}
\label{thm:scheme:SingIntgBound}
Let $\beta_0 \in (0,n)$ and suppose that
for a constant $c > 0$
small enough with respect to $n,k,m$,
\begin{align}
\label{eq:scheme:DimensionConditions}
	(k-1) n < m < kn,
	\quad
	\beta_0 \leq \beta < n,
	\quad
	n - c \beta_0 \leq \alpha < n.
\end{align}
Then
\begin{align}
\label{eq:scheme:SingIntgBound}
	\Lambda^*(\, |\wh{\mu}| , \dots, |\wh{\mu}| ; |J| \,)
	\lesssim_{\beta_0} (n - \alpha)^{-O(1)}.
\end{align}
\end{proposition}

Recalling Definition~\ref{thm:scheme:OpMeasures}, 
we see that $\Lambda(\mu,\dots,\mu)$ is well-defined
under the conditions~\eqref{eq:scheme:DimensionConditions}.
In practice, we will need slight variants of 
Proposition~\ref{thm:scheme:SingIntgBound},
which are discussed in Section~\ref{sec:intg}.
In the same section, 
we obtain singular integral bounds for bounded functions of compact support.

\begin{proposition}
\label{thm:scheme:OpFourierControl}
Suppose that $m > (k-1)n$.
Then there exists $\eps \in (0,1)$
depending at most on $n,k,m$
such that the following holds.
For functions $f_0,\dots,f_k \in \bump{n}$ 
with support in $\cube{2}$,
\begin{align*}
	|\Lambda(f_0,\dots,f_k)|
	\lesssim \prod_{0 \leq j \leq k} \| \wh{f}_j \|_{\infty}^\eps \cdot \| f_j \|_\infty^{1-\eps}.
\end{align*}
\end{proposition}

In Section~\ref{sec:config}, 
we construct a measure detecting polynomial configurations,
by exploiting the finiteness of the singular integral
in~\eqref{eq:scheme:SingIntgBound}
and the uniform decay of the fractal measure.

\begin{proposition}
\label{thm:scheme:ConfigMeasure}
Let $\beta_0 \in (0,n)$ and suppose 
that~\eqref{eq:scheme:DimensionConditions} holds.
Then there exists a measure $\nu \in \meas{n+m}$
such that
\begin{itemize}
	\item	$\| \nu \| = \Lambda( \mu, \dots, \mu )$,
	\item	$\nu$ is supported on the set of $(x,y) \in \R^n \times \Omega$ such that: 
			\newline
			$( x , x + \varphi_1(y), \dots, x + \varphi_k(y)) \in E^{k+1}$,
	\item	$\nu(H) = 0$ for every hyperplane $H < \R^{n + m}$.
\end{itemize}
\end{proposition}

In Section~\ref{sec:abs}, we show how to
obtain a positive mass of polynomial configurations
in sets of positive density, 
through the singular integral bound of 
Proposition~\ref{thm:scheme:OpFourierControl}
and the arithmetic regularity lemma from additive combinatorics.

\begin{proposition}
\label{thm:scheme:AbsContEst}
Suppose that $m > (k-1)n$.
Then, uniformly for every function
$f \in \bump{n}$ such that
$\Supp f \subset { \cube{8} }$, 
$0 \leq f \leq 1$ and $\int f = \tau \in (0,1]$,
we have
\begin{align*}
	\Lambda(f,\dots,f) \gtrsim_\tau 1.
\end{align*}
\end{proposition}

In Section~\ref{sec:transf},
we show how to obtain a positive mass of configurations
by a transference argument, by which the fractal measure $\mu$
is replaced by a mollified version of itself which is absolutely
continuous with bounded density, allowing us to invoke 
Proposition~\ref{thm:scheme:AbsContEst}.

\begin{proposition}
\label{thm:scheme:Transference}
Let $\beta_0 \in (0,n)$ and	
suppose that 
\begin{align*}
	(k-1)n < m < kn,
	\quad
	\beta_0 \leq \beta < n,
	\quad
	n - c(\beta_0) \leq \alpha < n,	
\end{align*}
for a sufficiently small constant $c(\beta_0) > 0$.
Then
\begin{align*}
	\Lambda( \mu, \dots, \mu ) > 0 .
\end{align*}
\end{proposition}

At this stage we have stated all the necessary ingredients to 
prove the main theorem.

\smallskip
\textit{Proof of Theorem~\ref{thm:intro:MainThm}.}
We may assume that $E \subset \cube{16}$
after a translation and dilation, which does
not affect the assumptions on $\mu,(A_i),Q$ except for the introduction
of constant factors in bounds.
By Proposition~\ref{thm:scheme:ConfigMeasure} ,
there exists a measure $\nu \in \meas{n + m}$ 
with mass $\Lambda(\mu,\dots,\mu)$ supported on
\begin{align*}
	X = \{ (x,y) \in \R^n \times \Omega \,:\,
	 (x, x + A_1 y, \dots, x + A_{k-1} y , x + A_k y + Q(y) e_n) \in E^{k+1} \},
\end{align*}
and such that $\nu(V_i) = 0$ for every collection of 
hyperplanes $V_1,\dots,V_q$ of $\R^{n+m}$.
We have therefore proven the result if we can show that
$\| \nu \| = \Lambda(\mu,\dots,\mu) > 0$, for then 
$\nu(X \smallsetminus (V_1 \cup \cdots \cup V_q)) > 0$
and the set $X \smallsetminus (V_1 \cup \cdots \cup V_q)$ cannot be empty.
We may apply Proposition~\ref{thm:scheme:Transference}
to obtain precisely this conclusion when $\alpha$ is close enough to $n$
with respect to $\beta_0$ (and the other implicit parameters $n,k,m,D,\bfA,Q$). 
\qed

\smallskip

To conclude this outline, we comment briefly
on the role that the Fourier decay hypothesis plays in our argument.
Using the restriction theory of fractals,
the assumption~\eqref{eq:scheme:FourierDecay}
is used together with the ball condition~\eqref{eq:scheme:BallDecay} in Appendix~\ref{sec:restr}
to deduce that $\| \wh{\mu} \|_{2+\eps} < \infty$ for an arbitrary $\eps > 0$, 
provided that $\alpha$ is close enough to $n$ (depending on $\eps$).
The Hausdorff dimension condition~\eqref{eq:scheme:BallDecay} alone does yield
information on the average Fourier decay of $\mu$, via the 
energy formula~\cite[Chapter~8]{Wolff:Book}, but this type of estimate
seems to be insufficient to establish the boundedness 
of the singular integrals we encounter.
Section~\ref{sec:intg} on singular integral bounds
and Section~\ref{sec:abs} on absolutely continuous estimates
only use the Fourier moment bound above.
On the other hand,
the estimation of degenerate configurations in Section~\ref{sec:config}
and the transference argument of Section~\ref{sec:transf}
exploit in an essential way the assumption of uniform Fourier decay.

\section{Counting operators and Fourier expressions}
\label{sec:ops}

In this section we describe the various types of
pattern-counting operators and singular integrals that
arise in trying to detect translation-invariant patterns
in the fractal set of the introduction.
First, we define an oscillatory integral which arises naturally
in the Fourier expression of the configuration form 
in Definition~\ref{thm:scheme:ConfigOp}.

\begin{definition}[Oscillatory integral]
\label{thm:ops:OscIntg}
For $\bfxi \in (\R^n)^k$ and $\theta \in \R^m$ we define
\begin{align*}
	J_\theta(\bfxi) = \int_{\R^m} e\big[ ( \theta + \bfA^\transp \bfxi) \cdot y + \xi_{kn} Q(y) \big] \psi(y) \dy,
	\qquad J = J_0.
\end{align*}
\end{definition}

We now derive the dual expression of the configuration form
announced in Section~\ref{sec:scheme}.

\smallskip
\textit{Proof of Proposition~\ref{thm:scheme:FourierInv}.}
By inserting the Fourier expansions of $f_1,\dots,f_k$ and by Fubini, we have
\begin{align*}
	\Lambda(f_0, \dots, f_k)
	&= \int_{\R^n} \int_{\R^m}
	f_0(x) f_1(x + \varphi_1(y)) \cdots f_k(x + \varphi_k(y)) 
	\dx \psi(y) \dy 
	\\
	&= \int_{(\R^n)^{k}} \wh{f}_1 ( \xi_1 ) \cdots \wh{f} ( \xi_k ) 
	\\
	& \phantom{\int_{(\R^n)^k}} 
	\int_{\R^n} f_0(x) e\big[ ( \xi_1 + \dotsb + \xi_k ) \cdot x \big] \dx 
	\\
	& \phantom{\int_{(\R^n)^k}} 
	\int_{\R^m} e\big[ \xi_1 \cdot \varphi_1(y) + \dotsb + \xi_k \cdot \varphi_k(y) \big] \psi(y) \dy 
	\ \dxi_1 \dots \dxi_k.
\end{align*}
Recalling Definition~\ref{thm:ops:OscIntg} and the choice~\eqref{eq:scheme:OnePolPattern}, 
we deduce that
\begin{align*}
	\Lambda(f_0, \dots, f_k)
	= \int_{(\R^n)^{k}} \wh{f}_0(-\xi_1-\dotsb-\xi_k) \wh{f}_1 (\xi_1) \cdots \wh{f}_k(\xi_k) J(\bfxi) 
	\dxi_1 \dots \dxi_k.
\end{align*}
\qed

We single out a useful bound for the configuration operator,
typically used when the $\lambda_i$ are either the measure $\mu$ 
or a mollified version of it.

\begin{proposition}
\label{thm:ops:OpBoundBySingIntg}
For measures $\lambda_0,\dots,\lambda_k \in \meas{n}$, we have
\begin{align*}
	|\Lambda(\lambda_0,\dots,\lambda_k)|
	\leq
	\prod_{j=0}^k \| \wh{\lambda}_j \|_\infty^\eps
	\cdot \Lambda^*\big(\, |\wh{\lambda}_0|^{1-\eps}, \dots, |\wh{\lambda}_k|^{1-\eps}; |J| \,),
\end{align*}
where the left-hand side is absolutely convergent if the right-hand side
is finite.
\end{proposition}

\begin{proof}
This follows from Definition~\ref{thm:scheme:OpMeasures}
and the successive bounds
\begin{align*}
	&\phantom{=.} |\Lambda^*(\wh{\lambda}_0,\dots,\wh{\lambda}_k ; J)|
	\\
	&\leq \int_{(\R^n)^k} |\wh{\lambda}_0(\xi_1 + \dotsb + \xi_k)| 
	|\wh{\lambda}_1(\xi_1)| \cdots |\wh{\lambda}_k(\xi_k)| |J(\bfxi)| \dbfxi
	\\
	&\leq \prod_{j=0}^k  \| \wh{\lambda}_j \|_\infty^\eps \cdot
	\int_{(\R^n)^k} |\wh{\lambda}_0(\xi_1 + \dotsb + \xi_k)|^{1-\eps}
	|\wh{\lambda}_1(\xi_1)|^{1-\eps} \cdots |\wh{\lambda}_k(\xi_k)|^{1-\eps} |J(\bfxi)| \dbfxi.
\end{align*}
\end{proof}

In some instances we will need 
a slightly more general multilinear form, as follows.

\begin{definition}[Smoothed configuration form]
\label{thm:ops:ConfigOpSmooth}
For functions $f_0,\dots,f_k \in \schw{n}$ and $F \in \schw{n+m}$, let
\begin{align}
\label{eq:ops:ConfigOpSmooth}
	\Lambda(f_0,\dots,f_k;F)
	= \int_{\R^n} \int_{\R^m}
	F(x,y) f_0(x) f_1(x + \varphi_1(y)) \cdots f_k(x + \varphi_k(y)) \dx\, \psi(y) \dy.
\end{align}
\end{definition}

\begin{proposition}
\label{thm:ops:FourierInvOpSmooth}
For functions $f_0,\dots,f_k \in \schw{n}$ and $F \in \schw{n+m}$, we have
\begin{align*}
	&\phantom{= .} \Lambda(f_0,\dots,f_k;F)
	&= \int_{\R^n \times \R^m} \wh{F}(\kappa,\theta)
	\int_{(\R^n)^k} \wh{f}_0( - \kappa - \xi_1 - \dotsb - \xi_k) 
	\prod_{j=1}^k \wh{f}_j (\xi_j) J_\theta(\bfxi) 
	\dbfxi \dkappa \dtheta.
\end{align*}
\end{proposition}

\begin{proof}
By inserting the Fourier expansions of $F,f_1,\dots,f_k$ and by Fubini,
we obtain 
\begin{align*}
	&\phantom{= .} \Lambda(f_0, \dots, f_k ; F ) \\
	&= 
	\int_{\R^n} \int_{\R^m}
	F(x,y) f_0(x) f_1(x + \varphi_1(y)) \cdots f_k(x + \varphi_k(y)) \dx \psi(y) \dy \\
	&= 
	\int_{\R^n \times \R^m} \wh{F}( \kappa, \theta ) 
	\int_{(\R^n)^k} \wh{f}_1 ( \xi_1 ) \cdots \wh{f} ( \xi_k )
	\\
	& \phantom{\int_{(\R^n)^k}} 
	\int_{\R^n} f_0(x) e\big[ ( \kappa + \xi_1 + \dotsb + \xi_k ) \cdot x \big] \dx \\
	& \phantom{\int_{(\R^n)^k}} 
	\int_{\R^m} e\big[ \theta \cdot y + \xi_1 \cdot \varphi_1(y) + \dotsb + \xi_k \cdot \varphi_k(y) \big] \psi(y) \dy
	\ \dxi_1 \dots \dxi_k \dkappa \dtheta
	\\
	&= 	\int_{\R^n \times \R^m} \wh{F}( \kappa, \theta ) 
	\int_{(\R^n)^k} \wh{f}_0( - \kappa - \xi_1 - \dotsb - \xi_k) \wh{f}_1(\xi_1) \cdots \wh{f}_k( \xi_k )
	J_\theta(\bfxi) \dbfxi \dkappa \dtheta.
\end{align*}
\end{proof}

\section{Bounding the singular integral}
\label{sec:intg}

This section is devoted to the central task
of bounding the singular integral~\eqref{eq:scheme:SingIntg},
when the kernel $K$ involved
is the oscillatory integral $J_\theta$
from Definition~\ref{thm:ops:OscIntg}.
We will rely crucially on the following decay estimate.

\begin{proposition}
Assuming that the neighborhood $\Omega$ of zero 
has been chosen small enough, we have
\begin{align}
\label{eq:intg:OscIntgDecay}
	&&
	|J_\theta(\bfxi)| &\lesssim ( 1 + |\bfA^\transp \bfxi + \theta| )^{-m/2}
	&&(\bfxi \in (\R^n)^k,\, \theta \in \R^m).
\end{align}
\end{proposition}

\begin{proof}
By Definition~\ref{thm:ops:OscIntg}, we have
$J_\theta(\bfxi) = I( \bfA^\transp \bfxi + \theta , \xi_{kn} )$,
where
\begin{align*}
	I( \gamma, \gamma_{m+1} )
	= \int_{\R^m} e( \gamma \cdot y + \gamma_{m+1} Q(y) ) \psi(y) \dx.
\end{align*}
Consider the hypersurface $S = \{ (y,Q(y)) \,:\, y \in \Supp(\psi) \}$ of $\R^{m+1}$,
then our assumptions on $Q$ mean that $S$ has non-zero Gaussian curvature.
Observe that $I$ is the Fourier transform of $\wt{\psi}\dsigma_S$,
where $\sigma_S$ is the surface measure on $S$ and $\wt{\psi}$ is a smooth function
with same support as $\psi$.
Therefore it satisfies the decay estimate~\cite[Chapter~VIII]{Stein:Book}
\begin{align*}
	|I( \gamma, \gamma_{m+1} )| \lesssim ( 1 + |\gamma| + |\gamma_{m+1}| )^{-m/2}
\end{align*}
uniformly in $(\gamma,\gamma_{m+1}) \in \R^{m+1}$, which concludes the proof.
\end{proof}

The main result of this section is
a bound on the singular integral
for functions in $L^s$, 
for a range of $s$ depending on $n,m,k$.
In practice we will apply 
the proposition below when $s$ is close to $2$,
which requires the parameter $m'$ to be larger than $(k-1)n$,
and when the functions $F_i$ are powers of $|\wh{\mu}|$ or
bounded functions supported on $\cube{2}$.

\begin{proposition}
\label{thm:intg:MainIntgBound}
Let $1 + \tfrac{1}{k} < s < k+1$ and $m' > 0$, 
and write 
\begin{align*}
	&\phantom{(\bfxi \in (\R^n)^k, \theta \in \R^m)} &
	K_{\theta,m'}(\bfxi) &= (1 + |\bfA^\transp \bfxi + \theta|)^{-m'/2}
	&& (\bfxi \in (\R^n)^k, \theta \in \R^m).
\end{align*}
Let $F_0,\dots,F_k$ be non-negative measurable functions on $\R^n$.
Provided that 
\begin{align}
\label{eq:intg:mEq}
	m' > 2kn - \frac{2(k+1)}{s} n,
\end{align}
we have, uniformly in $\theta \in \R^m$,
\begin{align*}
	\Lambda^*( F_0, \dots, F_k ; K_{\theta,m'} ) 
	\lesssim_{s,m'} \| F_0 \|_s \cdots \| F_k \|_s.
\end{align*}
\end{proposition}

The first step towards the proof of this proposition
is to bound moments of the kernels $K_{\theta,m'}$
on certain subspaces.
Consider the $k+1$ linear maps $(\R^n)^k \rightarrow \R^n$ given by 
\begin{align*}
	\bfxi \mapsto - (\xi_1 + \dotsb + \xi_k) \eqqcolon \xi_0,
	\qquad
	\bfxi \mapsto \xi_j
	\qquad
	(1 \leq j \leq k).
\end{align*}
For every $0 \leq j \leq k$ and $\eta \in \R^n$, 
the set $\{ \bfxi \in (\R^n)^k \,:\, \xi_j = \eta \}$
is an affine subspace of $(\R^n)^k$ of dimension $(k-1)n$.
Recall that $\bfA^\transp : \R^{nk} \rightarrow \R^m$,
so that in the regime $m \geq (k-1) n$ 
we expect $(1 + |\bfA^\transp \cdot |)^{-1}$
to have bounded moments of order $q > (k-1)n$ 
on each of the subspaces $\{ \xi_j = \eta \}$,
under reasonable non-degeneracy conditions on the matrix $\bfA$.
As the next lemma shows, what is needed is precisely the content of 
Definition~\ref{thm:scheme:NonDegen}.

\begin{proposition}
\label{thm:intg:FiberCdt}
Let $0 \leq j \leq k$ and suppose that $m \geq (k-1)n$.
Then for $q > (k-1)n$ we have,
uniformly in $\eta \in \R^n$ and $\theta \in \R^m$,
\begin{align*}
	\int_{\xi_j = \eta} (1 + |\bfA^\transp \bfxi + \theta|)^{-q} \, \dsigma(\bfxi) \lesssim_q 1.
\end{align*}
\end{proposition}

\begin{proof}
First note that the assumptions of Definition~\ref{thm:scheme:NonDegen}
mean that $\bfA^\transp$ is injective on $\{ \bfxi \,:\, \xi_j = 0 \}$ for $0 \leq j \leq k$.
To see that, observe that the conditions
\begin{align*}
	&\phantom{(1 \leq j \leq k)} &
	\bfA^\transp \bfxi = 0,\ \xi_j = 0 \quad&\Rightarrow\quad \bfxi = 0
	&&(0 \leq j \leq k)
\end{align*}
can be put in matrix form
\begin{align*}
	&\phantom{(1 \leq j \leq k)} &
	\begin{bmatrix}
		A_1^\transp & \hdots & A_j^\transp & \hdots & A_k^\transp	\\
		0 & \hdots & I_{n \times n} & \hdots & 0 
	\end{bmatrix} \bfxi = 0
	\quad&\Rightarrow\quad \bfxi = 0
	&&(1 \leq j \leq k),
	\\
	&&
	\begin{bmatrix}
		A_1^\transp & \hdots & A_k^\transp	\\
		I_{n \times n} & \hdots & I_{n \times n} 
	\end{bmatrix} \bfxi = 0 
	\quad&\Rightarrow\quad \bfxi = 0.
	&&
\end{align*}
Since $m + n \geq kn$, the $(m+n) \times kn$ matrices above
have empty kernel if and only if they have rank $kn$,
a set of conditions which is easily seen to be equivalent
to that of Definition~\ref{thm:scheme:NonDegen}.

Now let
\begin{align*}
	I = \int_{\xi_j = \eta} (1 + |\bfA^\transp \bfxi + \theta|)^{-q} \dsigma(\xi).
\end{align*}
We parametrize the affine subspace $\{ \xi_j = \eta \}$ by 
$\bfxi = \bfR \bfxi' + \bfxi_\eta$, where $\bfxi'$ runs over $(\R^n)^k$,
$\bfxi_\eta \in (\R^n)^k$ is picked such that $(\bfxi_\eta)_j = \eta$,
and $\bfR \in O(\R^{kn})$ is a rotation mapping the 
subspace $\R^{(k-1)n}$ to $\{ \xi_j = 0 \}$.
We obtain
\begin{align*}
	I = \int_{\R^{(k-1)n}} \Big(1 + | \bfA^\transp \bfR \bfxi' + \bfA^\transp \bfxi_\eta + \theta | \Big)^{-q} \dbfxi',
\end{align*}
and we write $\bfB = \bfA^\transp \bfR \in \R^{m \times kn}$, 
which is injective on $\R^{(k-1)n}$.
Consider the orthogonal decomposition 
$\bfA^\transp \bfxi_\eta + \theta = \bfB \bfxi_{\eta,\theta} + \gamma$
with $\bfxi_{\eta,\theta} \in \R^{(k-1)n}$ and 
$\gamma \in (\bfB(\R^{(k-1)n}))^\bot$, and observe that 
by Pythagoras and injectivity,
\begin{align*}
	| \bfB \bfxi' + \bfA^\transp \bfxi_\eta + \theta |
	= | \bfB( \bfxi' + \bfxi_{\eta,\theta} ) + \gamma | 
	\geq | \bfB (\bfxi' + \bfxi_{\eta,\theta}) |
	\gtrsim | \bfxi' + \bfxi_{\eta,\theta} |.
\end{align*}
Via the change of variables $\bfxi' \leftarrow \bfxi' + \bfxi_{\eta,\theta}$,
we deduce that
\begin{align*}
	I \lesssim \int_{\R^{(k-1)n}} \big( 1 + | \bfxi' | \big)^{-q} \mathrm{d}\bfxi',
\end{align*} 
which is bounded for $q > (k-1)n$, uniformly in $\eta \in \R^n$.
\end{proof}

\begin{proposition}
\label{thm:intg:IntermIntgBound}
	Let $F_0,\dots,F_k$ be non-negative measurable functions on $\R^n$.
	Let $\tau \in (0,1)$ and $p,p' \in (1,+\infty)$
	be parameters with $\frac{1}{p} + \frac{1}{p'} = 1$.	
	Let $H \geq 0$ be a parameter and suppose that
	$K$ is a non-negative measurable function on $\R^{nk}$ such that
	\begin{align*}
		&\phantom{(\eta \in \R^n,\ 0 \leq j \leq k)} &
		\int_{\xi_j = \eta} K(\bfxi)^{p'} \dsigma(\bfxi) &\leq H
		&&(\eta \in \R^n,\ 0 \leq j \leq k).
	\end{align*}
	Then
	\begin{align*}
		&\phantom{\leq .} \Lambda^*( F_0, \dots, F_k; K) \\
		&\leq H^{1/p'} \prod_{j=0}^k \bigg( \int_{\R^n} F_j(\eta)^{\tau p (k+1)/k} \deta \bigg)^{\frac{k}{k+1} \frac{1}{p}}
		\bigg( \int_{\R^n} F_j(\eta)^{(1-\tau)p'(k+1)} \deta \bigg)^{\frac{1}{k+1} \frac{1}{p'}}.
	\end{align*}
\end{proposition}

\begin{proof}
We write $I = \Lambda^*( F_0, \dots, F_k; K )$.
By a first application of Hölder:
\begin{align}
	\notag
	I &= \int_{(\R^n)^k} \prod_{j=0}^k F_j( \xi_j )^{\tau + (1 - \tau)} K(\bfxi) \dbfxi \\
	\notag
	&\leq \bigg[ \int_{(\R^n)^k} \bigg( \prod_{j=0}^k F_j( \xi_j ) \bigg)^{\tau p} \dbfxi \bigg]^{\frac{1}{p}} \times
	\bigg[	\int_{(\R^n)^k}	\bigg( \prod_{j=0}^k F_j( \xi_j ) \bigg)^{(1-\tau)p'} 
		K(\bfxi)^{p'} \dbfxi \bigg]^{\frac{1}{p'}} \\
	\label{eq:intg:Ibound}
	& \eqqcolon ( I_1 )^{\tfrac{1}{p}} \times ( I_2 )^{\tfrac{1}{p'}}.
\end{align}
We can rewrite $I_1$ as follows: 
\begin{align*}
	I_1 
	= \int\limits_{(\R^n)^k} \prod_{j=0}^k F_j(\xi_j)^{\tau p} \dbfxi 
	= \int\limits_{(\R^n)^k} \prod_{i=0}^k \Bigg[ \prod_{\substack{ 0 \leq j \leq k \\ j \neq i }}
	F_j(\xi_j)^{\tau p} \Bigg]^{\frac{1}{k}} \dbfxi.
\end{align*}
By Hölder, we can then reduce to integrals involving each only $k$ of the $\xi_j$'s:
\begin{align*}
	I_1
	\leq \prod_{i=0}^k \bigg[ \, \int\limits_{(\R^n)^k}
	\prod_{\substack{ 0 \leq j \leq k \\ j \neq i }} 
	F_j(\xi_j)^{ \tau p (k+1)/k}  \dbfxi \bigg]^{\frac{1}{k+1}}.	
\end{align*}
Recall that $\xi_0 = \xi_1 + \dotsb + \xi_k$, so that after
appropriate changes of variables, each inner integral splits and we have
\begin{align}
	\notag
	I_1 
	&\leq
	\prod_{i=0}^k \Bigg[
	\prod_{\substack{ 0 \leq j \leq k \\ j \neq i }}
	\int_{\R^n} F_j(\eta)^{\tau p (k+1)/k} \deta \Bigg]^{\frac{1}{k+1}} 
	\\
	\label{eq:intg:I1Bound}
	&= \prod_{j=0}^k \bigg( \int_{\R^n}  F_j(\eta)^{\tau p (k+1)/k} \deta \bigg)^{\frac{k}{k+1}}.
\end{align}
To treat the integral $I_2$, we separate variables by Hölder,
and then integrate along slices~\cite{Nicolaescu:Coarea}:
\begin{align*}
	I_2
	&= \int_{(\R^n)^k} \prod_{j=0}^k F_j(\xi_j)^{(1-\tau)p'} K(\xi)^{p'} \dbfxi \\
	&\leq \prod_{j=0}^k \bigg[ \int_{(\R^n)^k} F_j(\xi_j)^{(1-\tau)p'(k+1)} K(\bfxi)^{p'} \dbfxi \bigg]^{\frac{1}{k+1}} \\
	&= \prod_{j=0}^k \bigg[ \int_{\eta \in \R^n} F_j(\eta)^{(1-\tau)p'(k+1)}
		\bigg( \int_{\xi_j = \eta}  K(\bfxi)^{p'} \dsigma(\bfxi) \bigg) \deta \bigg]^{\frac{1}{k+1}}.
\end{align*}
Inside each inner integral we use the fiber moment condition, so that eventually
\begin{align}
\label{eq:intg:I2Bound}
	I_2 \leq H \prod_{j=0}^k \bigg( \int_{\R^n} F_j(\eta)^{(1-\tau)p'(k+1)} \deta \bigg)^{\frac{1}{k+1}}.
\end{align}
The proof is finished upon inserting~\eqref{eq:intg:I1Bound}
and~\eqref{eq:intg:I2Bound} into~\eqref{eq:intg:Ibound}.
\end{proof}

It remains to determine the parameters $(\tau,p)$
in Proposition~\ref{thm:intg:IntermIntgBound}
that lead to a bound involving a single $L^s$ norm.

\begin{corollary}
\label{thm:intg:SystEqs}
Suppose that $1 + \tfrac{1}{k} < s < k+1$. 
Then there exists unique parameters 
$\tau \in (0,1)$ and $p \in (1,\infty)$ 
depending on $k$ and $s$ such that
\begin{align}
	\label{eq:intg:eq1}
	s = \frac{k+1}{k} p\tau = (k+1) p' (1 - \tau),
\end{align}
where $\tfrac{1}{p} + \tfrac{1}{p'} = 1$, 
and for such $(\tau,p)$ we have
\begin{align}
	\label{eq:intg:eq2}
	\frac{k+1}{s} &= \frac{k}{p} + \frac{1}{p'},
	\\
	\label{eq:intg:eq3}
	\frac{1}{p'} &= \frac{1}{k-1}\Big( k - \frac{k+1}{s} \Big).
\end{align}
\end{corollary}

\begin{proof}
Starting from~\eqref{eq:intg:eq1},
and dividing by $\tfrac{k+1}{k}p$ in the first identity
and by $\tfrac{k+1}{k} pp'$ in the second,
we obtain the equivalent identities
\begin{align}
\label{eq:intg:eqtheta}
	\tau = \frac{k}{k+1} \frac{s}{p}
	\quad\text{and}\quad
	\Big( \frac{k}{p} + \frac{1}{p'} \Big) \tau = \frac{k}{p}.
\end{align}
Inserting the left-hand expression of $\tau$ 
in the right-hand identity,
we deduce the relation~\eqref{eq:intg:eq2}.
This is easily solved in $p,p'$ and one finds that
\begin{align*}
	\frac{1}{p} = \frac{1}{k-1} \Big( \frac{k+1}{s} - 1 \Big),
	\qquad
	\frac{1}{p'} = \frac{1}{k-1} \Big( k - \frac{k+1}{s} \Big),
\end{align*}
which in particular recovers~\eqref{eq:intg:eq3}.
It can be checked that $\tfrac{1}{p} \in (0,1)$
under the given conditions on $s$.
Inserting this value of $\frac{1}{p}$ in
the first identity of~\eqref{eq:intg:eqtheta}, we find that
\begin{align*}
	\tau = \frac{k}{k-1} \Big( 1 - \frac{s}{k+1} \Big),
\end{align*}
which again lies in $(0,1)$ for the given range of $s$.
\end{proof}

\textit{Proof of Proposition~\ref{thm:intg:MainIntgBound}.}
Apply Proposition~\ref{thm:intg:IntermIntgBound}
with $K( \bfxi ) = ( 1 + |\bfA^\transp \bfxi + \theta| )^{-m'/2}$,
and the choice of parameters $(\tau,p)$ from Proposition~\ref{thm:intg:SystEqs}.
By~\eqref{eq:intg:eq2}, this gives
\begin{align*}
	|\lambda^*( F_0, \dots, F_k; K )|
	&\leq H^{1/p'} \prod_{j=0}^k ( \| F_j \|_s^s )^{ \tfrac{1}{k+1} \big( \tfrac{k}{p} + \tfrac{1}{p'} \big) } \\
	&= H^{1/p'} \prod_{j=0}^k \| F_j \|_s,
\end{align*}
where $H = \max_{j} \sup_{\eta,\theta} \int_{\xi_j = \eta} ( 1 + |\bfA^\transp \bfxi + \theta| )^{-p'm'/2} \dsigma(\bfxi)$.
Via Proposition~\ref{thm:intg:FiberCdt} and~\eqref{eq:intg:eq3},
we have $H \lesssim_{s,m'} 1$ provided that
\begin{align*}
	m' > \frac{2(k-1)n}{p'} = 2 \Big( k - \frac{k+1}{s} \Big) n.
\end{align*}
\qed

From Proposition~\ref{thm:intg:MainIntgBound}, 
we now derive useful bounds on
the dual form $\Lambda^*$,
which are needed to develop the results of
Sections~\ref{sec:config}--\ref{sec:transf}.
In the course of the proof, we refer to a restriction estimate
from Appendix~\ref{sec:restr}, which states essentially 
that $\wh{\mu}$ is in $L^{2+\eps}$ when 
$\beta$ remains bounded away from zero
and $\alpha$ is close enough to $n$.
Recall the notation $T^\kappa f = f( \kappa + \,\cdot)$ 
from Section~\ref{sec:notation}.

\begin{proposition}
\label{thm:intg:SingIntgBounds}
Let $\beta_0 \in (0,n)$ and suppose that
for a constant $c > 0$
small enough with respect to $n,k,m$,
\begin{align*}
	(k-1) n < m < kn,
	\quad
	\beta_0 \leq \beta < n,
	\quad
	n - c \beta_0 \leq \alpha < n.
\end{align*}
Then there exists $\eps \in (0,1)$ 
depending at most on $n,k,m$ such that
\begin{align*}
	\sup\limits_{ (\kappa,\theta) \in \R^n \times \R^m }
	\Lambda^*(\, T^\kappa |\wh{\mu}|^{1-\eps}, |\wh{\mu}|^{1-\eps}, \dots, 
	|\wh{\mu}|^{1-\eps} ; |J_\theta|^{1-\eps} \,)
	< \infty,
	\\
	\Lambda^*(\, |\wh{\mu}|^{1-\eps}, \dots, |\wh{\mu}|^{1-\eps} ; |J| \,)
	\lesssim_{\beta_0} (n - \alpha)^{-O(1)}.
\end{align*}
\end{proposition}

\begin{proof}
Let $\eps, \delta \in (0,1)$ be parameters.
Recalling the majoration~\eqref{eq:intg:OscIntgDecay},
we apply Proposition~\ref{thm:intg:MainIntgBound}
to $F_0 = T^\kappa |\wh{\mu}|^{1-\eps}$ and 
$F_i = |\wh{\mu}|^{1-\eps}$ for $i \geq 1$, 
with parameters $m' = (1 - \eps) m$ and $s = \tfrac{2 + \delta}{1-\eps}$.
The condition~\eqref{eq:intg:mEq} is fulfilled when
$m > (k-1)n$ and $\eps,\delta$ are small enough 
with respect to $n,k,m$.
We obtain, uniformly in $\kappa \in \R^n$ and $\theta \in \R^m$,
\begin{align*}
	\Lambda^*( T^\kappa |\wh{\mu}|^{1-\eps}, |\wh{\mu}|^{1-\eps}, \dots, |\wh{\mu}|^{1-\eps}; |J_\theta|^{1-\eps} )
	\lesssim_{\eps,s} \| |\wh{\mu}|^{1-\eps} \|_s^{k+1}
	= \| \wh{\mu} \|_{2+\delta}^{(1-\eps)(k+1)},
\end{align*}
By Proposition~\ref{thm:restr:MuMoment} and~\eqref{eq:scheme:Dgrowth}, we conclude that
\begin{align*}
		\Lambda^*( T^\kappa |\wh{\mu}|^{1-\eps}, |\wh{\mu}|^{1-\eps}, \dots, |\wh{\mu}|^{1-\eps} ; |J_\theta|^{1-\eps} )
		\lesssim_{\eps,\delta,\beta_0} (n-\alpha)^{-O(1)},
\end{align*}
and the second bound follows since $|J| \lesssim |J|^{1-\eps}$.
\end{proof}

\smallskip

\textit{Proof of Proposition~\ref{thm:scheme:OpFourierControl}.}
Let $\eps \in (0,1)$ be a parameter.
By Proposition~\ref{thm:ops:OpBoundBySingIntg} and~\eqref{eq:intg:OscIntgDecay}, we have
\begin{align}
\label{eq:intg:OpPullingOut}
	|\Lambda(f_0,\dots,f_k)| 
	\leq 
	\textstyle
	\prod_{0 \leq j \leq k} \| \wh{f}_j \|_{\infty}^\eps \cdot
	\Lambda^*\big( |\wh{f}_0|^{1-\eps} , \dots , |\wh{f}_k|^{1-\eps} ; (1 + |\bfA^\transp \,\cdot|)^{-m/2} \big).
\end{align}
For $\eps$ small enough with respect to $n,k,m$,
we may apply Proposition~\ref{thm:intg:MainIntgBound} 
with $s = \frac{2}{1-\eps}$ and $m' = m$,
together with Plancherel:
\begin{align*}
	\Lambda^*( |\wh{f}_0|^{1-\eps} , \dots , |\wh{f}_k|^{1-\eps} ; (1 + |\bfA^\transp \,\cdot|)^{-m/2} )
	&\lesssim  \textstyle \prod_{j=0}^k \| |\wh{f}_j|^{1-\eps} \|_{2/(1-\eps)}  \\
	&= \textstyle \prod_{j=0}^k \| \wh{f}_j \|_2^{1-\eps}	\\
	&= \textstyle \prod_{j=0}^k \| f_j \|_2^{1-\eps}	\\
	&\leq \textstyle \prod_{j=0}^k \| f_j \|_\infty^{1-\eps},
\end{align*}
where we used the assumption $\Supp(f_j) \subset \cube{2}$ in the last line.
Inserting this bound in~\eqref{eq:intg:OpPullingOut} 
finishes the proof.
\qed

\section{The configuration measure}
\label{sec:config}

In this section,
we aim to construct the measure $\nu \in \meas{n + m}$ 
specified in Proposition~\ref{thm:scheme:ConfigMeasure}.
We make extensive use of the singular integral bounds
derived in the previous section.
Our treatment is similar to that of
Chan et al.~\cite{CLP:Configs}, but we work in
a more abstract setting.
We assume throughout this section that
the dimensionality conditions~\eqref{eq:scheme:DimensionConditions} are met,
so that singular integral bounds are available.

We start with the proper definition of $\nu$,
which is the content of the next proposition
(recall Definition~\ref{thm:scheme:ConfigOp} and
Proposition~\ref{thm:scheme:FourierInv}).
We define an extra shift function $\varphi_0 = 0$
for notational convenience.

\begin{proposition}
\label{thm:config:nuDef}
	Define the functional $\nu$ at $F \in \schw{n + m}$ by
	\begin{align*}
		\langle \nu, F \rangle
		= \lim_{\eps \rightarrow 0} \Lambda( \mu_\eps, \dots, \mu_\eps ; F ),
	\end{align*}
	where $\mu_\eps = \mu \ast \phi_\eps$ for an approximate identity $\phi_\eps$
	with $\phi \in \cutoff{n}$.
	Then $\nu$ is well-defined and we have, for every $F \in \schw{n+m}$,
	\begin{align*}
		\langle \nu, F \rangle &=
		\Lambda^*( \wh{\mu}, \dots, \wh{\mu} ; \wh{F} ), \\
		| \langle \nu, F \rangle | &\leq \| F \|_{\infty} \Lambda( \mu , \dots, \mu ),
	\end{align*}
	where the integrals defined by the right-hand side expressions converge absolutely.
	Therefore $\nu$ extends by density to a positive bounded 
	linear operator on $\mathcal{C}_c(\R^{n+m})$.
\end{proposition}

\begin{proof}
By Proposition~\ref{thm:ops:FourierInvOpSmooth}, we have
\begin{align*}
	&\phantom{= .} \Lambda( \mu_\eps, \dots, \mu_\eps ; F ) \\
	&= \int_{\R^{n+m}} \wh{F}(\kappa,\theta)
	\int_{(\R^n)^k} \wh{\mu}( - \kappa - \xi_1 - \dotsb - \xi_k ) \prod_{j=1}^k \wh{\mu}(\xi_j)
	J_\theta(\bfxi) h_\eps(\bfxi,\kappa) \dbfxi \dkappa\dtheta
\end{align*}
where $h_\eps(\bfxi,\kappa) = \wh{\phi}( - \eps (\kappa + \xi_1 + \dotsb + \xi_k ) ) \prod_{j=1}^k \wh{\phi}(\eps \xi_j)$.
Since $h_\eps$ is bounded by one in absolute value
and tends to $1$ pointwise as $\eps \rightarrow 0$,
the limit of $\Lambda( \mu_\eps, \dots, \mu_\eps ; F )$ as $\eps \rightarrow 0$ exists 
and equals $\Lambda^*( \wh{\mu}, \dots, \wh{\mu} ; \wh{F} )$ by dominated convergence,
since we have uniform boundedness of
\begin{align*}
	&\phantom{=} \int_{\R^{n+m}} |\wh{F}(\kappa,\theta)|
	\int_{(\R^n)^k}|\wh{\mu}(\kappa + \xi_1 + \dotsb + \xi_k)|
	\prod_{j=1}^k  |\wh{\mu}(\xi_j)|
	|J_\theta(\bfxi)| \dbfxi \dkappa\dtheta
	\\
	&\leq
	\sup\limits_{(\kappa,\theta) \in \R^n \times \R^m} 
	\Lambda^*(\, |T^\kappa \wh{\mu}|, |\wh{\mu}|, \dots, |\wh{\mu}| ; |J_\theta| \,) \times
	\int_{\R^{n+m}} |\wh{F}(\kappa,\theta)| \dkappa\dtheta 
	< \infty,
\end{align*}
via Proposition~\ref{thm:intg:SingIntgBounds} 
and the majorations 
$|J_\theta| \lesssim |J_\theta|^{1-\eps}$, $|\wh{\mu}| \leq |\wh{\mu}|^{1-\eps}$.
Recalling Definitions~\ref{thm:scheme:ConfigOp} and~\ref{thm:ops:ConfigOpSmooth}, 
and using the positivity of $\mu_\eps$, we have also
\begin{align*}
	|\langle \nu, F \rangle|
	= \lim_{\eps \rightarrow 0} |\Lambda( \mu_\eps, \dots, \mu_\eps ; F )|
	\leq \| F \|_\infty \, \cjg{\lim\limits_{\eps \rightarrow 0}}\, \Lambda( \mu_\eps, \dots, \mu_\eps ).
\end{align*}
By Fourier inversion (Proposition~\ref{thm:scheme:FourierInv}) 
and another instance of the dominated convergence theorem,
exploiting the finiteness of $\Lambda^*( |\wh{\mu}|, \dots, |\wh{\mu}| ; |J| )$
provided by Proposition~\ref{thm:intg:SingIntgBounds},
we obtain
\begin{align*}
	|\langle \nu, F \rangle|
	\leq \| F \|_\infty \, \cjg{\lim\limits_{\eps \rightarrow 0}}\, 
	\Lambda^*( \wh{\mu}_\eps, \dots, \wh{\mu}_\eps ; J )
	= \| F \|_\infty \, \Lambda^*( \wh{\mu}, \dots, \wh{\mu} ; J ).
\end{align*}
This last quantity equals $\| F \|_\infty\, \Lambda( \mu, \dots, \mu )$ by 
Definition~\ref{thm:scheme:OpMeasures}.
\end{proof}

\begin{proposition}
\label{thm:config:nuSupport}
	When defined, the measure $\nu$ 
	of Proposition~\ref{thm:config:nuDef} 
	is supported on the compact set
	\begin{align*}
		X = \{ (x,y) \in E \times \Supp \psi \,:\,
		( x, x + \varphi_1(y) , \dots , x + \varphi_k(y) ) \in E^{k+1} \}.
	\end{align*}
\end{proposition}

\begin{proof}
We can rewrite $X = ( E \times \Supp \psi ) \cap \Phi^{-1}(E^{k+1})$,
where $\Phi$ is the smooth map defined by~\eqref{eq:scheme:ShiftsPattern},
so that $X$ is closed and bounded, and therefore compact.
Since its complement $X^c$ is open, it is enough to show
that $\langle \nu, F \rangle = 0$ for
every $F \in \cutoff{n+m}$ such that
$\Supp F \subset X^c$.
By compactness we know that
there exists $c > 0$ such that
\begin{align*}
	\max_{0 \leq j \leq k} d(x + \varphi_j(y),E) \geq c > 0 
	\quad\text{for all}\quad 
	(x,y) \in \Supp F \cap (\R^n \times \Supp \psi).
\end{align*}
On the other hand,
\begin{align*}
	\langle \nu, F \rangle
	= \lim_{\eps \rightarrow 0} \int_{\Supp F \cap (\R^n \times \Supp \psi)} 
	F(x,y) \prod_{j=0}^k \mu_\eps(x + \varphi_j(y)) \dx\, \psi(y) \dy.
\end{align*}
For $\eps$ small enough, 
since $\mu_\eps$ is supported on $E + B(0,C\eps)$
for a certain $C > 0$,
the integrand above is always zero.
\end{proof}

\begin{proposition}
\label{thm:config:nuMass}
We have $\| \nu \| = \Lambda( \mu, \dots, \mu )$.
\end{proposition}

\begin{proof}
Consider the compact set $X$ from Proposition~\ref{thm:config:nuSupport},
and the larger compact set
\begin{align*}
	Y = \{ (x,y) \in \R^n \times \Supp \psi \,:\, 
			d(x + \varphi_j(y),E) \leq 1 \quad \text{for $0 \leq j \leq k$} \}.
\end{align*}
Pick a smoothed ball indicator $F \in \cutoff{n+m}$ 
such that $F = 1$ on $Y$.
Since $\nu$ is supported on $X \subset Y$, we have
\begin{align*}
	\nu(\R^{n+m}) 
	= \langle \nu, F \rangle
	= \lim_{\eps \rightarrow 0} \int_{\R^n \times \Supp \psi} F(x,y) 
	\prod_{j = 0}^k \mu_\eps( x + \varphi_j(y) ) \dx \psi(y) \dy.
\end{align*}
Since $(x,y) \mapsto \prod_{j = 0}^k \mu_\eps( x + \varphi_j(y) )$ is supported on $Y$
for $\eps$ small enough, we have therefore
\begin{align*}
	\nu(\R^{n+m}) 
	= \lim_{\eps \rightarrow 0}	\Lambda( \mu_\eps, \dots, \mu_\eps ).
\end{align*}
By the same reasoning as in the end of the proof of 
Proposition~\ref{thm:config:nuDef}, 
using again the bound $\Lambda^*(\, |\wh{\mu}| , \dots , |\wh{\mu}| ; |J| \,) < \infty$
provided by Proposition~\ref{thm:intg:SingIntgBounds},
we find eventually that
$\| \nu \| = \Lambda( \mu, \dots, \mu )$.
\end{proof}

We now turn to the last expected feature of the
configuration measure $\nu$, which is that
it has zero mass on any hyperplane.

\begin{proposition}
\label{thm:config:nuZeroMassHp}
We have $\nu(H) = 0$ for every hyperplane $H$ of $\R^{n+m}$.
\end{proposition}

\begin{proof}
Consider a hyperplane $H < \R^{n+m}$
and a rotation $R \in O_{n+m}(\R)$
such that $H = R ( \R^{n+m-1} \times \{0\} )$.
Consider parameters $L \geq 1$ and $\delta \in (0,1]$.
We consider a Schwartz function $F_\delta$ of the form
\begin{align}
\label{eq:config:HpCutoff}
	F_\delta \circ{R} = \chi\Big( \frac{\cdot}{L} \Big) \, \Xi\Big( \frac{\cdot}{\delta} \Big),
\end{align}
where $\chi \in \schw{n+m-1}$ and $\Xi \in \mathcal{S}(\R)$
are non-negative such that $\chi \geq 1$ on $[-1,1]^{n+m-1}$
and $\Xi(0) \geq 1$.
Writing $H_L = R (\, [-L,L]^{n+m-1} \times \{0\} \,)$, 
we have therefore
$\nu(H_L) \leq \langle \nu , F_\delta \rangle$,
and it is enough to show that $\langle \nu , F_\delta \rangle$
tends to $0$ as $\delta \rightarrow 0$,
for every fixed $L \geq  1$.
By Proposition~\ref{thm:config:nuDef},
and writing $\gamma = (\kappa,\theta) \in \R^n \times \R^m$, we have
\begin{align}
\label{eq:config:MassHpCutoff}
	\langle \nu, F_\delta \rangle
	= \int\limits_{\R^{n+m}} \int\limits_{(\R^n)^k} \wh{F}_\delta(\gamma)
	\wh{\mu}(- \kappa - \xi_1 - \dotsb - \xi_k) \prod_{j=1}^k \wh{\mu}(\xi_j)
	J_\theta( \bfxi ) \dbfxi \mathrm{d}\gamma.
\end{align}

We assume that $\chi$, $\Xi$ have been chosen 
so that their Fourier transforms
are supported on centered balls of radius $1$,
which is certainly possible.
Recalling~\eqref{eq:config:HpCutoff},
we have therefore,
for every $(u,v) \in \R^{n+m-1} \times \R$,
\begin{align}
\label{eq:config:HpCutoffDecay}
	|\wh{F}_\delta \circ R (u,v) | 
	= |\wh{F_\delta \circ R}(u,v)|
	\lesssim L^{n+m-1} \cdot 1_{|u| \leq L^{-1}} \cdot \delta \cdot 1_{|v| \leq \delta^{-1}}.
\end{align}
We next show how to obtain some uniform $\gamma$-decay from the other factor in the integrand 
of~\eqref{eq:config:MassHpCutoff}.
By~\eqref{eq:scheme:FourierDecay} and~\eqref{eq:intg:OscIntgDecay}, 
and since $\beta \leq n \leq m$, we have
\begin{align*}
	\notag
	&\phantom{= .} | \wh{\mu}( \kappa + \xi_1 + \dots + \xi_k ) |
	\textstyle \prod_{j=1}^k |\wh{\mu}(\xi_j)| |J_\theta(\bfxi)|
	\\
	\notag
	&\lesssim_{\alpha} ( 1 + |\kappa + \xi_1 + \dotsb + \xi_k | )^{-\tfrac{\beta}{2}}
	\textstyle \prod_{j=1}^k ( 1 + |\xi_j| )^{-\tfrac{\beta}{2}} ( 1 + |\bfA^\transp \bfxi + \theta| )^{-\tfrac{m}{2}} 
	\\
	\notag
	&\lesssim_{\alpha} ( 1 + |\kappa + \xi_1 + \dotsb + \xi_k |
	+ |\xi_1| + \dotsb + |\xi_k| + |\bfA^\transp \bfxi + \theta| )^{-\tfrac{\beta}{2}}.
\end{align*}
Using the above in cunjunction with the triangle inequality
and the decompositions 
$\theta = ( \bfA^\transp \bfxi + \theta ) - \sum_{j=1}^k A_j^\transp \xi_j$
and $\kappa = (\kappa + \xi_1 + \dotsb + \xi_k) - \sum_{j=1}^k \xi_j$,
we deduce that
\begin{align}
	\notag
	&\phantom{= .} | \wh{\mu}( \kappa + \xi_1 + \dots + \xi_k ) |
	\textstyle \prod_{j=1}^k |\wh{\mu}(\xi_j)| |J_\theta(\bfxi)|
	\\
	\notag
	&\lesssim_{\alpha} ( 1 + |\kappa| + |\theta| )^{-\tfrac{\beta}{2}}
	\\
	\label{eq:config:MeasureDecay}
	&\asymp ( 1 + |\gamma| )^{-\tfrac{\beta}{2}}.
\end{align}

Let $\eps \in (0,1)$ be the small parameter in the
statement of the proposition.
At this point we have two parametrizations
$\gamma = (\kappa,\theta) = R(u,v)$ with 
$(\kappa,\theta) \in \R^n \times \R^m$, $(u,v) \in \R^{n+m-1} \times \R$.
By integrating in $(u,v)$-coordinates in~\eqref{eq:config:MassHpCutoff},
and bounding $\wh{F}_\delta(\gamma)$ via~\eqref{eq:config:HpCutoffDecay}, 
we obtain
\begin{align*}
	|\langle \nu, F_\delta \rangle| 
	&\lesssim
	\int_{\R^{n + m - 1} \times \R} 1_{|u| \leq L^{-1}} \cdot L^{n+m-1} \cdot \delta \cdot 1_{|v| \leq \delta^{-1}}
	\\
	&\phantom{\lesssim \int} \bigg[ \int_{(\R^n)^k} |\wh{\mu}( \kappa + \xi_1 + \dotsb + \xi_k )|
	\prod_{j=1}^k |\wh{\mu}(\xi_j)| |J_\theta(\bfxi)| \dbfxi \bigg] \du \dv \\
\end{align*}
By pulling out an $\eps$-th power of the inner integrand 
and using~\eqref{eq:config:MeasureDecay}, we infer that
\begin{align*}
	&\phantom{=} |\langle \nu, F_\delta \rangle| 
	\\
	&\lesssim_\alpha
	\int_{\R^{n+m-1}} L^{n+m-1} \cdot 1_{ |u| \leq L^{-1} } 
	\int_{\R} \delta \cdot 1_{ |v| \leq \delta^{-1} } \cdot ( 1 + |(u,v)| )^{- \tfrac{\eps\beta}{2} } \\
	&\phantom{\lesssim_\alpha \int} \bigg[ \int_{(\R^n)^k} |\wh{\mu}( \kappa + \xi_1 + \dotsb + \xi_k )|^{1-\eps}
	\prod_{j=1}^k |\wh{\mu}(\xi_j)| ^{1-\eps} |J_\theta(\bfxi)|^{1-\eps} \dbfxi \bigg] \du \dv 
	\\
	&\lesssim \sup\limits_{ (\kappa,\theta) \in \R^n \times \R^m} 
	\Lambda^*(\, |T^\kappa \wh{\mu}|^{1-\eps}, |\wh{\mu}|^{1-\eps}, \dots, |\wh{\mu}|^{1-\eps} ; |J_\theta|^{1-\eps} \,)
	\times \delta \int_{ |v| \leq \delta^{-1} } (1 + |v|)^{ - \tfrac{\eps\beta}{2}} \dv.
\end{align*}
The supremum above is finite by Proposition~\ref{thm:intg:SingIntgBounds},
and for $\eps$ small enough, the last factor is bounded by
$\delta^{\eps\beta/2}$. 
Therefore $\langle \nu, F_\delta \rangle \rightarrow 0$ as $\delta \rightarrow 0$,
as was to be shown.
\end{proof}

\smallskip
\textit{Proof of Proposition~\ref{thm:scheme:ConfigMeasure}.}
It suffices to combine Propositions~\ref{thm:config:nuDef}--~\ref{thm:config:nuZeroMassHp},
recalling that we assumed~\eqref{eq:scheme:DimensionConditions}
in this section.
\qed

\section{Absolutely continuous estimates}
\label{sec:abs}

In this section we verify that absolutely continuous
estimates are available when the shifts in~\eqref{eq:scheme:ShiftsPattern}
are given by polynomial vectors and the singular integral converges.
We work with the notation of abstract shift functions.

The strategy, as in the regularity proof of Roth's theorem~\cite{Tao:U2RegLemma}, 
is to use the $U^2$ arithmetic regularity lemma 
to decompose a non-negative bounded function into
an almost-periodic component, an $L^2$ error,
and a part which is Fourier-small.
The precise version of the regularity lemma 
that we need is found in Appendix~\ref{sec:reg}.
To neglect the contribution of Fourier-small functions, 
we use the fact that the counting operator is controlled
by the Fourier $L^\infty$ norm for bounded functions, 
in the sense of Proposition~\ref{thm:scheme:OpFourierControl}.
To show that the pattern count for almost-periodic functions is high,
we need uniform lower bounds for certain Bohr sets of almost-periods,
the proof of which will occupy subsequent parts of this section.
We define a Bohr set of $\T^n$
of frequency set $\Gamma \subset \Z^n$,
radius $\delta \in \big( 0,\frac{1}{2} \big]$ and dimension $d = |\Gamma| < \infty$ by
\begin{align}
\label{eq:abs:BohrSetDef}
	B = B( \Gamma, \delta )
	= \{ x \in \T^n \,:\, \| \xi \cdot x \| \leq \delta \quad \forall\,\xi \in \Gamma \}.
\end{align}
We first prove the following conditional version 
of Proposition~\ref{thm:scheme:AbsContEst}.

\begin{proposition}
\label{thm:abs:AbsContEst}
Suppose that $m > (k-1)n$ and, 
uniformly for every Bohr set $B$ of $\T^n$
of dimension $d$ and radius $\delta > 0$,
\begin{align*}
	\leb\big\lbrace\, y \in [-c,c]^m \,:\, \varphi_1(y),\dots,\varphi_k(y) \in B \,\big\rbrace \gtrsim_{d,\delta} 1.
\end{align*}
Then for every function	$f \in \bump{n}$
supported on $\cube{8}$
such that $0 \leq f \leq 1$
and $\int f = \tau \in (0,1]$, we have
\begin{align*}
	\Lambda(f,\dots,f) \gtrsim_{\tau} 1.
\end{align*}
\end{proposition}

\begin{proof}
We let $\kappa : (0,1]^3 \rightarrow (0,1]$ be a decay function
and $\eps \in (0,1]$ be a parameter, both to be determined later.
Write the decomposition of Proposition~\ref{thm:reg:RegLemma}
with respect to $\eps,\kappa$
as $f = f_1 + f_2 + f_3 = g + f_3$.
Note that $f_1,g \geq 0$ and
$f_1,f_2,f_3,g$
are supported in $\cube{4}$ and
uniformly bounded by $2$ in absolute value.
Expanding $f = g + f_3$ by multilinearity, 
and using Proposition~\ref{thm:scheme:OpFourierControl}
together with the Fourier bound on $f_3$ in~\eqref{eq:reg:RegLemmaEqs},
we obtain
\begin{align}
	\notag
	\Lambda(f,\dots,f) 
	&= \Lambda(g,\dots,g)
	 + O( \, \textstyle \sum \Lambda( * , \dots, f_3 , \dots, * ) \, ) \\
	 \label{eq:abs:FirstRegDcp}
	&= \Lambda(g,\dots,g) + O( \kappa( \eps, d^{-1}, \delta )^{\eps'} ),
\end{align}
for a certain $\eps' \in (0,1)$ depending at most on $n,k,m$.
Recall that we assumed that $\psi$ is at least $1$ 
on a box $[-c,c]^m$ in Section~\ref{sec:scheme}, and let
\begin{align}
\label{eq:abs:EDef}
	E = \big\lbrace\, y \in [-c,c]^m \,:\, \varphi_1(y),\dots,\varphi_k(y) \in B \,\big\rbrace,
\end{align}
where $B$ is the Bohr set of Proposition~\ref{thm:reg:RegLemma}.
For reasons that shall be clear later, we first restrict
integration to the set $E$, 
using the non-negativity of $g$:
\begin{align*}
	\Lambda(g,\dots,g)
	\geq \int_E \bigg( \, \int_{\R^n} g \cdot T^{\varphi_1(y)} g \cdots T^{\varphi_k(y)} g \, \dleb \bigg) \dy.
\end{align*}
Next, we focus on the decomposition $g = f_1 + f_2$ 
and exploit the $L^2$ bound on $f_2$ in~\eqref{eq:reg:RegLemmaEqs}
by Cauchy-Schwarz in the inner integral:
\begin{align*}
	\Lambda(g,\dots,g)
	&\geq \int_E \bigg( \, \int_{\R^n} g \cdot T^{\varphi_1(y)} g \cdots T^{\varphi_k(y)} g \, \dleb \bigg) \dy \\
	&\geq \int_E \bigg( \, 
			\int_{\R^n} f_1 \cdot T^{\varphi_1(y)} f_1 \cdots T^{\varphi_k(y)} f_1 \, \dleb
			- \sum \int_{\R^n} * \cdots T^{\varphi_j(y)} f_2 \cdots * \, \dleb \bigg) \dy \\
	&\geq \int_E \bigg( \, \int_{\R^n} 
			f_1 \cdot T^{\varphi_1(y)} f_1 \cdots T^{\varphi_k(y)} f_1 \, \dleb - O(\eps) \bigg) \dy .
\end{align*}
Finally, we use the almost-periodicity estimate for $f_1$ in~\eqref{eq:reg:RegLemmaEqs} 
and the definition~\eqref{eq:abs:EDef} of $E$
to replace the shifts of $f_1$ by itself:
\begin{align*}
	\Lambda(g,\dots,g)
	\geq \int_E \bigg( \, \int_{[-\frac{1}{2},\frac{1}{2}]^n} f_1^{k+1} \dleb - O( \eps ) \bigg) \dy
\end{align*}
By nesting of $L^p\big(\cube{2}\big)$ norms and non-negativity of $f_1$, we infer that
\begin{align*}
	\Lambda(g,\dots,g)
	&\geq \int_E \bigg( \, \bigg( \int_{[-\frac{1}{2},\frac{1}{2}]^n} f_1 \dleb \bigg)^{k+1} - O( \eps ) \bigg) \dy \\
	&= \leb(E) \cdot ( \tau^{k+1} - O( \eps ) ).
\end{align*}
Choosing $\eps = c \tau^{k+1}$ with a small $c > 0$, 
and recalling~\eqref{eq:abs:FirstRegDcp}
and the assumption on $E$, we obtain
\begin{align*}
	\Lambda(f,\dots,f) \geq c(\delta,d^{-1}) \tau^{k+1} - O( \kappa(c\tau^{k+1},d^{-1},\delta)^{\eps'} ).
\end{align*}
Choosing $\kappa(\eps,d^{-1},\delta) = c' \cdot \big( c(\delta,d^{-1}) \eps \big)^{1/\eps'}$, 
we obtain
\begin{align*}
	\Lambda(f,\dots,f) \geq \tfrac{1}{2} c(\delta,d^{-1}) \tau^{k+1}
	\gtrsim_{\tau} 1,
\end{align*}
recalling that $d,\delta^{-1} \lesssim_{\eps,\kappa} 1 \lesssim_{\tau} 1$.
\end{proof}

It remains to determine a lower bound on the measure 
of the intersection of preimages of a Bohr set 
by the shift functions.
This can be done when the shift functions are
polynomial vectors,
by reduction to a known diophantine
approximation problem,
and in fact there will be a series of 
intermediate reductions.
We let $d$ denote the $L^\infty$ metric on $\R^n$ or $\R$ 
and we define
\begin{align*}
	\| x \|_{\T^n} = d( x , \Z^n ) = \max\limits_{1 \leq i \leq n} d(x_i,\Z)
\end{align*}
for $x \in \R^n$.
In all subsequent propositions in this section
we also liberate the letters $n$, $k$, $m$ 
from their usual meaning,
and we indicate the dependencies of implicit constants
in all parameters.
Our objective is to prove the following statement.

\begin{proposition}
\label{thm:abs:LowBoundBohrPolVecs}
	Let $t,m,n,\ell,d \geq 1$.
	Let $Q_1,\dots,Q_t : \R^m \rightarrow \R^n$ be polynomial vectors
	with components of degree at most $\ell$,
	and such that $Q_i(0) = 0$ for all $i \in [t]$.
	For $\xi_1,\dots,\xi_d \in \R^n$, we have
	\begin{align*}
		\leb\big\lbrace\, y \in [-c,c]^m \,:\, \| Q_i(y) \cdot \xi_j \|_{\T} < \eps 
		\quad\forall (i,j) \in [t] \times [d] \,\big\rbrace \gtrsim_{\,\eps,\ell,m,t,d,n} 1.
	\end{align*}
\end{proposition}

Our first reduction is to a finite system of conditions
on monomials modulo one.

\begin{proposition}
\label{thm:abs:LowBoundBohrMonom}
	Let $\ell,m \geq 1$ and $X = \{0,\dots,\ell\}^m \smallsetminus \{0\}$. 
	For every $I \in X$, let $d_I \in \Nzero$ and $\xi_I \in \R^{d_I}$.  
	We have\footnote{
	Here and in the sequel we set 
	$\R^0 = \{0\}$ and $\| 0 \|_{\T^0} = 0$
	so that the conditions involving a space $\R^{0}$ are void.}
	\begin{align*}
		\leb\big\lbrace\, y \in [-c,c]^m \,:\, \| y^I \xi_I \|_{\T^{d_I}} \leq \eps 
		\quad\forall I \in X \,\big\rbrace
		\gtrsim_{\,\eps,\ell,m,(d_I)} 1.		
	\end{align*}
\end{proposition}

\smallskip
\textit{Proof that Proposition~\ref{thm:abs:LowBoundBohrMonom}
implies Proposition~\ref{thm:abs:LowBoundBohrPolVecs}.}

We let $X=\{0,\dots,\ell\}^m \smallsetminus \{0\}$ 
and we write $Q_i = \sum_{k \in [n]} Q_{ik} e_k$,
$Q_{ik} = \sum_{I \in X} a_I^{(ik)} y^I$.
For every $I \in X$ we define $d_I = t + d + n$ and
$\xi_I = (a_I^{(ik)} \xi_{jk})_{(i,j,k)} \in \T^{t+d+n}$,
to make the following observation:
\begin{align*}
	&&
	&\phantom{\Leftrightarrow} \| Q_i(y) \cdot \xi_j \|_{\T} \leq \eps 	
	&&\forall (i,j) \in [t] \times [d]
	\\
	&&
	&\Leftrightarrow \textstyle
	\| \sum_{k \in [n]} \sum_{I \in X} a_I^{(ik)} y^I \xi_{jk} \|_{\T} \leq \eps 
	&&\forall (i,j) \in [t] \times [d]
	\\
	&&
	&\Leftarrow
	\| y^I a_I^{(ik)} \xi_{jk} \|_{\T} \leq \frac{\eps}{n\ell^m}  
	&&\forall (i,j,k) \in [t] \times [d] \times [n], \, I \in X
	\\
	&&
	&\Leftrightarrow
	\| y^I \xi_I \|_{\T^{d_I}} \leq \frac{\eps}{n\ell^m}
	&&\forall I \in X.
\end{align*}
Applying Proposition~\ref{thm:abs:LowBoundBohrMonom}
with $\eps \leftarrow \eps / n\ell^m$ and $(d_I,\xi_I)$ as above, 
we find a lower bound on the quantity under study
which depends only on $\eps,\ell,m,t,d,n$.
\qed
\smallskip

Our second reduction consists in a straightforward induction 
which reduces the dimension of the problem to $1$.

\begin{proposition}
\label{thm:abs:LowBoundBohrR}
	Let $\ell \geq 1$ and $d_1,\dots,d_\ell \in \Nzero$,
	$\xi_1 \in \R^{d_1},\dots,\xi_\ell \in \R^{d_\ell}$.
	We have
	\begin{align*}
		\leb\big\lbrace\, y \in [-c,c] \,:\, \| y^j \xi_j \|_{\T^{d_j}} \leq \eps 
		\quad\forall j \in [\ell] \,\big\rbrace
		\gtrsim_{\,\eps,\ell,(d_i)} 1 .
	\end{align*}
\end{proposition}

\textit{Proof that Proposition~\ref{thm:abs:LowBoundBohrR} implies
Proposition~\ref{thm:abs:LowBoundBohrMonom}.}

We induct on $m \geq 1$, the case $m = 1$ being
precisely Proposition~\ref{thm:abs:LowBoundBohrR}.
Assume that we have proven the estimate for dimensions less than or equal to $m$,
and write a tuple $I \in \{ 0,\dots,\ell\}^{m+1} \smallsetminus \{0\}$
as $I = (J,i_{m+1})$ with $J \in \{0,\dots,\ell\}^m$ 
and $i_{m+1} \geq 0$.
We distinguish the conditions
involving $y_{m+1}$ or not by Fubini:
\begin{align*}
	&\phantom{=} 
	\leb\big\lbrace\, y \in [-c,c]^{m+1} \,:\, 
	\| y^I \xi_I  \|_{\T^{d_I}} \leq \eps 
	\quad\forall\, I \in X\,  \big\rbrace 
	\\
	&= \int_{[-c,c]^{m+1}}
	1\Big[\, \| y^J y_{m+1}^{i_{m+1}} \xi_I  \|_{\T^{d_I}} \leq \eps  
	\quad\forall\, (J,i_{m+1}) = I \in X \,\Big] \dy_1 \dots \dy_m \dy_{m+1} 
	\\
	&= \int_{[-c,c]^m} 
	1\Big[\, \| y^J \xi_I  \|_{\T^{d_I}} \leq \eps 	
	\quad\forall (J,0) = I \in X \,\Big] 
	\\
	&\quad \int_{[-c,c]}
	1\Big[\, \| y_{m+1}^{i_{m+1}} \cdot y^{J} \xi_I  \|_{\T^{d_I}} \leq \eps 	
	\quad\forall (J,i_{m+1}) = I \in X\,:\, i_{m+1} \geq 1 \,\Big] 
	\dy_{m+1}\, \dy_1 \dots \dy_m.
\end{align*}
By first applying the induction hypothesis with $m=1$ at fixed $y_1,\dots,y_m$,
and then by applying another instance of the induction hypothesis, 
we find that this quantity is indeed bounded from below by a positive constant
depending only on $\eps$, $\ell$, $m$ and $(d_I)$.
\qed
\smallskip

Our final reduction is a simple discretization argument
which reduces the problem to the following 
known diophantine approximation 
estimate~\cite[Proposition~B.2]{LM:DiophApprox}
(see also~\cite[Proposition~A.2]{GT:DiophApprox},~\cite[Chapter~7]{Baker:Book}).

\begin{proposition}
\label{thm:abs:LowBoundBohrZ}
	Let $\ell \geq 1$ and $d_1,\dots,d_\ell \in \Nzero$.
	Let $\alpha_1 \in \R^{d_1},\dots,\alpha_\ell \in \R^{d_\ell}$ and $N \geq 1$.
	We have
	\begin{align*}
		N^{-1} \#\big\lbrace |n| \leq N \,:\, \| n^j \alpha_j \|_{\T^{d_j}} \leq \eps 
		\quad\forall j \in [\ell] \big\rbrace
		\gtrsim_{\,\eps,\ell,(d_j)} 1 .
	\end{align*}
\end{proposition}

\textit{Proof that Proposition~\ref{thm:abs:LowBoundBohrZ} implies
Proposition~\ref{thm:abs:LowBoundBohrR}.}

Consider a scale $N \geq 1$ going to infinity.
Write each $|y| \leq c$ as $y = \frac{n+u}{N}$ 
with $n \in \Z$ and $u \in \big( \! -\tfrac{1}{2},\tfrac{1}{2} \big]$,
so that
$y^j = n^j/N^j + O_\ell( 1/ N )$
for every $j \in [\ell]$.
For $N$ large enough with respect to $(\xi_j)$, $\eps$ and $\ell$,
we have therefore
\begin{align*}
	\| y^j \xi_j \|_{\T} \leq \eps
	\quad\Leftarrow\quad 	
	\Big\| n^j \frac{\xi_j}{N^j} \Big\|_{\T} \leq \frac{\eps}{2}.
\end{align*}
This yields:
\begin{align*}
	&\phantom{=}
	\leb\big\lbrace\, y \in [-c,c] \,:\, \| y^j \xi_j \|_{\T^{d_j}} \leq \eps 
	\quad\forall j \in [\ell] \,\big\rbrace  &\phantom{\sum_k}
	\\
	&\geq
	\sum_{ |n| \leq cN / 2 }
	\leb\big\lbrace\, y = \tfrac{n+u}{N} \,:\, |u| \leq \tfrac{1}{2},\ 
	\Big\| n^j \frac{\xi_j}{N^j} \Big\|_{\T^{d_j}} \leq \eps/2 
	\quad\forall j \in [\ell] \,\big\rbrace \\
	&\geq N^{-1}
	\#\big\lbrace |n| \leq cN/2 \,:\, \Big\| n^j \frac{\xi_j}{N^j} \Big\|_{\T^{d_j}} \leq \eps
		\quad\forall j \in [\ell] \,\big\rbrace.
\end{align*}
Applying Proposition~\ref{thm:abs:LowBoundBohrZ} concludes the proof.
\qed
\smallskip

To conclude this section we may now 
derive the absolutely continuous estimates
stated in Section~\ref{sec:scheme}.

\smallskip

\textit{Proof of Proposition~\ref{thm:scheme:AbsContEst}.}
It suffices to combine 
Propositions~\ref{thm:abs:AbsContEst}
and~\ref{thm:abs:LowBoundBohrPolVecs},
recalling the shape~\eqref{eq:scheme:OnePolPattern} of
our shift functions.
\qed

\section{The transference argument}
\label{sec:transf}

This section is concerned with proving that $\Lambda(\mu,\dots,\mu) > 0$,
by the transference argument of \L{}aba and Pramanik~\cite{LP:Configs}
exploiting the pseudorandomness
of the fractal measure $\mu$ as $\alpha \rightarrow n$.
We start by recalling the decomposition of Chan et al.~\cite[Section~6]{CLP:Configs}
of the fractal measure $\mu$
into a bounded smooth part (a mollified version of $\mu$) 
and a Fourier-small part (the difference with the first part).
This is the part of the argument where one lets
$\alpha$ tend to $n$ in a certain sense,
and then the Fourier tail exhibits very strong,
exponential-type decay in $n - \alpha$.

\begin{proposition}
\label{thm:transf:MeasDcp}
There exists a constant $C_1 > 0$ depending at most on $n,D$
and a decomposition $\mu = \mu_1 + \mu_2$,
where $\mu_1 = f\dleb$, 
$f \in \bump{n}$, $0 \leq f \leq C_1$, 
$\int f = 1$, $\Supp f \subset \cube{8}$,
$|\wh{\mu}_i| \leq 2 |\wh{\mu}|$ for $i \in \{ 1 , 2 \}$ and
\begin{align*}
	\| \wh{\mu}_2 \|_\infty
	\lesssim 
	(n - \alpha)^{-O(1)} e^{-\tfrac{\beta}{2+\beta} \cdot \tfrac{1}{n-\alpha}}.
\end{align*}
\end{proposition}

\begin{proof}
Let $L \geq 1$ be a parameter.
Consider a cutoff $\phi \in \bump{n}$ such that $\int \phi = 1$,
$\Supp \phi \subset B(0,\tfrac{1}{16})$
and $0 \leq \phi \leq C_0$,
for a certain $C_0 = C_0(n) > 0$,
and define $\phi_L = L^n \phi( L \,\cdot\, )$.
Let $f = \mu \ast \phi_L$
and consider the decomposition $\mu = \mu_1 + \mu_2$ 
with $\mu_1 = f \dleb$ and $\mu_2 = \mu - \mu_1$.
We can already infer that $f \geq 0$, $\int f = 1$, $|\wh{\mu}_i| \leq 2|\wh{\mu}|$
for $i=1,2$ and $\Supp \mu_1 \subset \cube{8}$, 
since we assumed that $E \subset \cube{16}$ in Section~\ref{sec:scheme}.

Next, we show that $f$ is bounded.
Since $\phi_L$ has support in $B(0,\tfrac{1}{16L})$,
by~\eqref{eq:scheme:BallDecay} we have
\begin{align*}
	f(x)
	&= \int_{B(x,\frac{1}{16L})} \phi_L( x - y ) \dmu(y) \\
	&\leq \| \phi_L \|_\infty \cdot \mu\big[ B(x,\tfrac{1}{16L}) \big] \\
	&\leq C_0 D L^{n - \alpha}. \phantom{\int_B}
\end{align*}
Choosing $L = e^{\frac{1}{n-\alpha}}$, 
we deduce that
\begin{align*}
	\| f \|_\infty \leq C_0 D e \eqqcolon C_1.
\end{align*}

Finally, we bound the Fourier transform of $\wh{\mu}_2$. 
Observe that, for every $\xi \in \R^n$,
\begin{align}
\label{eq:transf:Mu2}
	\wh{\mu}_2(\xi) = \wh{\mu}(\xi)\Big(1 - \wh{\phi}\Big(\frac{\xi}{L}\Big) \Big).
\end{align}
Since $\int \phi = 1$, we always have $|1 - \wh{\phi}(\tfrac{\xi}{L})| \leq 2$.
On the other hand, since $\phi$ has support in $B(0,\tfrac{1}{16})$, we have
\begin{align*}
	\Big| 1 - \wh{\phi}\Big( \frac{\xi}{L} \Big) \Big|
	= \bigg\lvert \int_{B(0,\frac{1}{16})} \phi(x) \Big( 1 - e\Big( \frac{\xi \cdot x}{L} \Big) \Big) \dx \bigg\rvert
	\lesssim \frac{|\xi|}{L}.
\end{align*}
By inserting these two last bounds 
in~\eqref{eq:transf:Mu2}, we obtain
\begin{align*}
	|\wh{\mu}_2(\xi)| 
	\lesssim \min\Big( 1, \frac{|\xi|}{L} \Big) |\wh{\mu}(\xi)|.
\end{align*}
Consequently, by~\eqref{eq:scheme:FourierDecay} and~\eqref{eq:scheme:Dgrowth} we have
\begin{align*}
	|\wh{\mu}_2(\xi)| 
	\lesssim 
	(n-\alpha)^{-O(1)}	
	\min\Big( 1, \frac{|\xi|}{L} \Big) 
	\min( 1 , |\xi|^{-\beta/2} ).
\end{align*}
By considering separately the ranges 
$|\xi| \geq L^{2/(2+\beta)}$ and $|\xi| \leq L^{2/(2+\beta)}$, we find that
\begin{align*}
	|\wh{\mu}_2(\xi)| \lesssim (n - \alpha)^{-O(1)} L^{- \tfrac{\beta}{2+\beta}}.
\end{align*}
Recalling our choice of $L$, we have
\begin{align*}
	|\wh{\mu}_2(\xi)| 
	\lesssim 
	(n - \alpha)^{-O(1)} e^{-\tfrac{\beta}{2+\beta} \cdot \tfrac{1}{n-\alpha}}.
\end{align*}
\end{proof}

We now establish the positivity of $\Lambda(\mu,\dots,\mu)$,
using the previous decomposition, 
with the main contribution from the absolutely continuous part
estimated by Proposition~\ref{thm:scheme:AbsContEst},
and the other contributions bounded away by
Proposition~\ref{thm:intg:SingIntgBounds}.

\textit{Proof of Proposition~\ref{thm:scheme:Transference}.}
We consider the decomposition $\mu = \mu_1 + \mu_2$
from Proposition~\ref{thm:transf:MeasDcp}, and expand by
multilinearity in
\begin{align*}
	\Lambda( \mu, \dots, \mu ) =
	C_1^{-(k+1)} \Lambda( \mu_1/C_1, \dots, \mu_1/C_1 ) 
	+ O\big(\, \textstyle \sum \Lambda( *, \dots, \mu_2 , \dots, * ) \,\big),
\end{align*}
where the sum is over $2^{k+1}-1$ terms and the stars
denote measures equal to either $\mu_1$ or $\mu_2$.
By Proposition~\ref{thm:scheme:AbsContEst}, we deduce that
for a certain constant $c > 0$, we have
\begin{align*}
	\Lambda( \mu, \dots, \mu ) \geq
	c - O\big(\, \textstyle \sum \Lambda( *, \dots, \mu_2 , \dots, * ) \,\big).
\end{align*}

By Proposition~\ref{thm:ops:OpBoundBySingIntg},
we have furthermore, for any $\eps \in (0,1)$,
\begin{align*}
	\Lambda( \mu, \dots, \mu )
	\geq c - O\big( \, \| \wh{\mu}_2 \|_\infty^{\eps}
	\, \Lambda^*(\, |\wh{\mu}|^{1-\eps},\dots,|\wh{\mu}|^{1-\eps} ; |J| \,) \, \big).
\end{align*}
By taking $\eps$ to be that appearing in Proposition~\ref{thm:intg:SingIntgBounds},
and inserting the Fourier bound on $\mu_2$ from Proposition~\ref{thm:transf:MeasDcp},  
we find that
\begin{align*}
	\Lambda( \mu, \dots, \mu )
	\geq c - O_{\beta_0}\Big( (n - \alpha)^{-O(1)} e^{- \eps \cdot \tfrac{\beta_0}{2 + \beta_0} \cdot \tfrac{1}{n - \alpha}} \Big),
\end{align*}
where we used the monotonicity of $x/(2+x)$.
This can be made positive for $\alpha \geq n - c(\beta_0,\eps)$
with $c(\beta_0,\eps) > 0$ small enough.
\qed

\section{Revisiting the linear case}
\label{sec:lin}

In this section we indicate how the method
of this article may be modified to obtain
the following extension of Theorem~\ref{thm:intro:CLP},
which allows for any positive exponent of Fourier decay
for the fractal measure.
For simplicity we only treat the case where $n$ divides $m$,
which already covers all the geometric applications
discussed in~\cite{CLP:Configs}.

\begin{theorem}
\label{thm:lin:CLPagain}
Let $n,k,m \geq 1$,
$D \geq 1$ and $\alpha, \beta \in (0,n)$.
Suppose that $E$ is a compact subset of $\R^n$
and $\mu$ is a probability measure supported on $E$
such that
\begin{align*}
	\mu\big[ B(x,r) \big]
	\leq Dr^\alpha
	\quad\text{and}\quad
	|\wh{\mu}(\xi)| \leq D(n-\alpha)^{-D} (1+|\xi|)^{-\beta/2}
\end{align*}
for all $x \in \R^n$, $r >0$ and $\xi \in \R^n$. 
Suppose that $(A_1,\dots,A_k)$ is a non-degenerate system
of $n \times m$ matrices
in the sense of Definition~\ref{thm:intro:NonDegen}.
Assume finally that
$m = (k-r) n$ with $1 \leq r < k$
and, for some $\beta_0 \in (0,n)$,
\begin{align*}
	\frac{k-1}{2} n < m < kn,
	\qquad
	\beta_0 \leq \beta < n,
	\qquad
	n - c_{n,k,m,\beta_0,D,(A_i)} \leq \alpha < n,
\end{align*}
for a sufficiently small constant $c_{n,k,m,\beta_0,D,(A_i)} > 0$.
Then, for every collection of strict subspaces $V_1,\dots,V_q$ of $\R^{n + m}$,
there exists $(x,y) \in \R^{n + m} \smallsetminus V_1 \cup \dots V_q$
such that
\begin{align*}
	( x , x + A_1 y , \dots, x + A_k y ) \in E^{k+1}.
\end{align*}
\end{theorem}

Note that the condition on $m$ is equivalent
to that of Theorem~\ref{thm:intro:CLP}.
We only sketch the proof of Theorem~\ref{thm:lin:CLPagain},
since it follows by a straightforward adaption
of the methods of this paper, with the only difference
lying in the treatment of the singular integral.

We start by stating a slight generalization of
Hölder's inequality that was already used (for $\ell = k+1$, $r = k$) in the
proof of Proposition~\ref{thm:intg:IntermIntgBound}.
We write $\binom{[\ell]}{r}$ for the set of subsets
of $[\ell]$ of size $r$.

\begin{proposition}
\label{thm:lin:GenHolder}
Let $(X,\mathfrak{M},\lambda)$ be a measured space 
and let $1 \leq r \leq \ell$.
For measurable functions $F_1,\dots,F_\ell : X \rightarrow \C$,
we have
\begin{align*}
	\int_X \prod_{j \in [\ell]} |F_j| \,\dlambda
	\leq \prod_{ S \in \binom{[\ell]}{r} }
	\Bigg[ \int_X \prod_{j \in S} |F_j|^{\ell/r} \dlambda \Bigg]^{1/\binom{\ell}{r}}.
\end{align*}
\end{proposition}

\begin{proof}
First observe that, for arbitrary real numbers $a_1,\dots a_\ell \geq 0$, we have
\begin{align*}
	\prod_{j \in [\ell]} a_j = 
	\prod_{ S \in \binom{[\ell]}{r} }
	\bigg( \prod_{j \in S} a_j \bigg)^{1/\binom{\ell-1}{r-1}}.
\end{align*}
Next, let $I = \int_X \prod_{j \in [\ell]} |F_j| \,\dlambda$ and
apply Hölder's inequality in
\begin{align*}
	I &= \int_X 
	\prod_{ S \in \binom{[\ell]}{r} }
	\bigg( \prod_{j \in S} |F_j| \bigg)^{1/\binom{\ell-1}{r-1}} \dlambda \\
	&\leq 
	\prod_{ S \in \binom{[\ell]}{r} }	
	\Bigg[ \int_X \bigg( \prod_{j \in S} |F_j| \bigg)^{\binom{\ell}{r} / \binom{\ell-1}{r-1}} \dlambda \Bigg]^{1/\binom{\ell}{r}}.
\end{align*}
A quick computation shows that $\binom{\ell}{r} / \binom{\ell-1}{r-1} = \ell/r$, which concludes the proof.
\end{proof}

We now place ourselves under the assumptions of Theorem~\ref{thm:lin:CLPagain},
and in particular we assume that the matrices $A_1,\dots,A_k$ 
are non-degenerate in the sense of Definition~\ref{thm:intro:NonDegen}.
We also write $A_0 = 0_{n \times n}$ throughout.
This matches the framework of this paper except
that now $Q = 0$.

We fix a smooth cutoff $\psi \in \cutoff{n}$ which is at least $1$ on a box $[-c,c]^n$.
We define the oscillatory integral
\begin{align}
\label{eq:lin:OscIntgDef}
	J(\bfxi) = \int_{\R^n} e(\bfA^\transp \bfxi \cdot y) \psi(y) \dy = \wh{\psi}( -\bfA^\transp \bfxi ).
\end{align}
The counting operators are now defined by\footnote{
In fact, one could work without cutoff functions in the $y$-variable,
as was done in~\cite{CLP:Configs}, which simplifies the
estimates somewhat. Here we keep smooth cutoffs to
stay closer to the framework of the article.}
\begin{align*}
	\Lambda(f_0,\dots,f_k)
	&= \int_{\R^n} \int_{\R^m} f_0(x) f_1(x + A_1 y) \cdots f_k(x + A_k y) \dx \ \psi(y) \dy, \\
	\Lambda^*(F_0, \dots, F_k ; J)
	&= \int_{(\R^n)^k} F_0(- \xi_1 - \dotsb - \xi_k) F_1( \xi_1 ) \cdots F_k( \xi_k ) J(\bfxi) \dbfxi,
\end{align*}
for functions $f_i,F_i \in \schw{n}$,
and we have $\Lambda(f_0,\dots,f_k) = \Lambda^*(\wh{f}_0,\dots,\wh{f}_k ;J)$ as before.

Since we assumed that $\psi\in \bump{m}$, it follows from~\eqref{eq:lin:OscIntgDef} that
\begin{align}
\label{eq:lin:OscIntgDecay}
	|J(\bfxi)| \lesssim_N (1 + |\bfA^\transp \bfxi|)^{-N}
\end{align}
for every $N > 0$.
Via some matricial considerations (as in~\cite[Lemma~3.2]{CLP:Configs}),
it can be checked that Definition~\ref{thm:intro:NonDegen}
is equivalent to the requirement that
$\bfA^\transp : \R^{kn} \rightarrow \R^{kn - rn}$
is injective on each subspace of the form
\begin{align*}
	\{ \bfxi \in (\R^n)^k \,:\, (\xi_j)_{j \in S} = \bfeta \},
\end{align*}
where $S$ is a subset of $\{0,\dots,k\}$ of size $r$
and $\bfeta \in \R^{rn}$,
and we wrote $\xi_0 = -( \xi_1 + \dotsb + \xi_k )$ as before.
Now consider an arbitrary subset $S$ of $\{0,\dots,k\}$ of size $r$.
By~\eqref{eq:lin:OscIntgDecay} one quickly deduces that
\begin{align}
\label{eq:lin:OscIntgFiberMoments}
	&\phantom{(q > 0\,\bfeta \in (\R^n)^S)} &
	\int_{(\xi_j)_{j \in S} = \bfeta} |J(\bfxi)|^q \dsigma(\bfxi) &\lesssim_q 1
	&&(q > 0,\,\bfeta \in (\R^n)^r),
\end{align}
in the same manner as in the proof of
Proposition~\ref{thm:intg:FiberCdt}.

In our linear setting one
may naturally obtain a better range of $m$ 
for which the multilinear form $\Lambda^*$ is controlled by $L^s$ norms.
The next proposition demonstrates this,
and it is applicable to our problem only when
when $\tfrac{k+1}{r} > 2 $, 
or equivalently $m = (k-r) n > \tfrac{k-1}{2} n$.

\begin{proposition}
\label{thm:lin:SingIntgBound}
We have
\begin{align*}
	|\Lambda^*(F_0,\dots,F_k;J)| \lesssim
	\| F_0 \|_{(k+1)/r} \cdots \| F_k \|_{(k+1)/r}.
\end{align*}
\end{proposition}

\begin{proof}
Write $I = \Lambda^*(F_0,\dots,F_k;J)$ and $[0,k] = \{ 0, \dots, k \}$
for the purpose of this proof.
By Proposition~\ref{thm:lin:GenHolder}, we have
\begin{align*}
	I
	&\leq \int_{(\R^n)^k} \prod_{j=0}^k \big( F_j(\xi_j) \cdot |J(\bfxi)|^{\frac{1}{k+1}} \big) \dbfxi \\
	&\leq \prod_{ S \in \binom{[0,k]}{r} }
	\bigg[ \int_{(\R^n)^k} \prod_{j \in S} F_j(\xi_j)^{\frac{k+1}{r}} |J(\bfxi)|^\frac{1}{r} \dbfxi \bigg]^{1/\binom{k+1}{r}}.
\end{align*}
Integrating along slices, and invoking~\eqref{eq:lin:OscIntgFiberMoments}, we obtain
\begin{align*}
	I
	&\leq \prod_{ S \in \binom{[0,k]}{r} }
	\bigg[ \int_{ (\R^n)^r } \prod_{j \in S} F_j (\eta_j)^{\frac{k+1}{r}} 
	\bigg( \int_{(\xi_j)_{j \in S} = \bfeta} |J(\bfxi)|^\frac{1}{r} \dsigma(\bfxi) \bigg) \dbfeta \bigg]^{1/\binom{k+1}{r}} 
	\\
	&\lesssim \prod_{ S \in \binom{[0,k]}{r} }
	\bigg[ \int_{ (\R^n)^r } \prod_{j \in S} F_j (\eta_j)^{\frac{k+1}{r}} \dbfeta \bigg]^{1/\binom{k+1}{r}} .
\end{align*}
Therefore each inner integral splits and we have
\begin{align*}
	I
	&\lesssim \prod_{ S \in \binom{[0,k]}{r} }
	\Bigg[ \prod_{j \in S} \int_{\R^n} F_j(\eta)^{\frac{k+1}{r}} \dbfeta \Bigg]^{1/\binom{k+1}{r}} \\
	&= \prod_{j \in [0,k]} 
	\bigg[ \int_{\R^n} F_j(\eta)^{\frac{k+1}{r}} \dbfeta \bigg]^{\binom{k}{r-1}/\binom{k+1}{r}}.
\end{align*}
Since $\binom{k+1}{r} / \binom{k}{r-1} = (k+1)/r$,
it follows that $I \leq \prod_{j \in [0,k]} \| F_j \|_{\frac{k+1}{r}}$, as was to be shown.
\end{proof}

With Proposition~\ref{thm:lin:SingIntgBound} in hand,
it is a simple matter to adapt the rest of the argument in this paper.
In fact, one would need a slight variant 
of that proposition involving a shift $\theta$,
as in the case of Proposition~\ref{thm:intg:MainIntgBound}.
From such a proposition one may deduce the natural analogues of
Propositions~\ref{thm:intg:SingIntgBounds}
and \ref{thm:scheme:OpFourierControl}, which will impose
the same conditions on $\alpha$ and $\beta$,
and a distinct condition $m > \frac{k-1}{2}n$ on $m$.
With these singular integral bounds in hand, 
the arguments of Sections~\ref{sec:config}--\ref{sec:transf} 
go through essentially unchanged,
and one obtains Theorem~\ref{thm:lin:CLPagain} 
by the process described at the end of Section~\ref{sec:scheme}.

\appendix

\section{The arithmetic regularity lemma}
\label{sec:reg}

In this section, we derive a version of
the $U^2$ arithmetic regularity lemma
following Tao's argument~\cite{Tao:U2RegLemma},
with minor twists to accomodate functions
defined over $\R^n$ instead of $\T^n$.
This set of ideas itself originates in a paper 
of Bourgain~\cite{Bourgain:DilatesConfig}, albeit in
a rather different language.
We include the complete proof since the exact result we need
is not stated in a convenient form in the literature.

We defined a Bohr set of $\T^n$
of frequency set $\Gamma \subset \Z^n$,
radius $\delta \in \big( 0,\frac{1}{2} \big]$ and dimension $d = |\Gamma| < \infty$ 
by~\eqref{eq:abs:BohrSetDef}.
We define the dilate of a Bohr set $B$ of frequency set $\Gamma$ and radius $\delta$
by a factor $\rho \in (0,1]$ as $B(\Gamma,\delta)_\rho = B(\Gamma,\rho\delta)$.
Note that $B(\Gamma,\delta) = \phi^{-1}( 2\delta \cdot Q )$
for the cube $Q = \cube{2}$
and the morphism $\phi : \R^n \rightarrow \T^d$, $x \mapsto (\xi \cdot x)_{\xi \in \Gamma}$.
We can find a cube covering of the form $Q \subset \bigcup_{t \in T} ( t + \delta \cdot Q )$
with $|T| = \lceil 1/\delta \rceil^d \leq (2/\delta)^d$, and therefore
\begin{align*}
	1 = |\phi^{-1}(Q)|
	\leq \sum_{t \in T} |\phi^{-1}( t + \delta \cdot Q)|.
\end{align*}
By the pigeonhole principle,
there exists $t \in T$ such that
$|\phi^{-1}( t + \delta \cdot Q )| \geq (\delta/2)^d$, 
and since
$\phi^{-1}( t + \delta \cdot Q ) - \phi^{-1}( t +  \delta \cdot Q ) \subset B$,
we deduce that
\begin{align}
\label{eq:reg:BohrSetLowerBound}
	|B| = |B(\Gamma,\delta)| \geq (\delta/2)^d
	\quad\text{for all $\delta \in (0,\tfrac{1}{2}]$}.
\end{align}

Now consider the tent function $\Delta(x) = (1 - |x|)^+$ on $\R$,
which is $1$-Lipschitz, bounded by $1$ everywhere,
and bounded from below by $1/2$ on $\big[\! -\tfrac{1}{2},\tfrac{1}{2} \big]$.
For any Bohr set $B$, 
we define functions $\phi_B, \nu_B : \T^n \rightarrow \C$ by
\begin{align*}
	\phi_B(x) =
	\Delta\Big( \frac{1}{\delta} \sup\limits_{\xi \in \Gamma} \| \xi \cdot x \| \Big),
	\quad
	\nu_B = \frac{\phi_B}{\int \phi_B},
\end{align*}
so that $\int \nu_B = 1$ and
$\tfrac{1}{2} 1_{B_{1/2}} \leq \nu_B \leq 1_B$.
The function $\nu_B$ is essentially
a smoothed normalized indicator function of the Bohr set $B$,
and its most important properties are summarized in the following proposition.

\begin{proposition}
For any Bohr set $B$ of frequency set $\Gamma \subset \Z^n$ 
and radius $\delta \in (0,\tfrac{1}{2}]$, 
we have
\begin{align}
	\label{eq:reg:BohrLinftyBound}
	&&
	\| \nu_B \|_\infty 
	&\lesssim (\delta/4)^{-d},
	&&
	\\
	\label{eq:reg:BohrAlmostPeriodicity}
	&&	
	\| T^t \nu_{B} - \nu_B \|_{\infty}
	&\lesssim (\delta/4)^{-d} \rho
	&&\text{for $t \in B_\rho$, $\rho \in (0,1]$,}
	\\
	\label{eq:reg:BohrAnnihilation}
	&&
	\wh{\nu}_B(\xi) &= 1 + O(\delta)
	&&\text{for $\xi \in \Gamma$}.
\end{align}
\end{proposition}

\begin{proof}
Note that $\int \phi_B \geq \tfrac{1}{2} |B_{1/2}| \geq \tfrac{1}{2}(\delta/4)^d$
by~\eqref{eq:reg:BohrSetLowerBound},
which implies the first estimate.
For every $x,t \in \T^n$, we also have
\begin{align*}
	| \nu_{B}(x + t) - \nu_B(x) |
	\leq 2(\delta/4)^{-d} \Big| \Delta\Big( \frac{1}{\delta}  \sup\limits_{\xi \in \Gamma} \| \xi \cdot (x+t) \| \Big) 
	- \Delta\Big( \frac{1}{\delta}  \sup\limits_{\xi \in \Gamma} \| \xi \cdot x \| \Big)  \Big|.
\end{align*}
When $t \in B_\rho$,  we have 
$\| \xi \cdot t \| \leq \rho\delta$ for every $\xi \in \Gamma$,
and therefore $| \nu_{B}(x+t) - \nu_B(x) | \lesssim (\delta/4)^{-d} \rho$
since $\Delta$ is $1$-Lipschitz,
and we have established the second estimate.
To obtain the third,
consider $\xi \in \Gamma$,
and observe that since $\nu_B$ is supported on $B$
and $\| \xi \cdot x \| \leq \delta$ for $x \in B$, we have
\begin{align*}
	\wh{\nu}_B(\xi) 
	= \int_B \nu_B(x) e( -\xi \cdot x ) \dx
	= (1 + O(\delta)) \cdot \int \nu_B
	= 1 + O(\delta).
\end{align*}
\end{proof}

\begin{proposition}
\label{thm:reg:RegLemma}
Let $\eps \in (0,1]$ be a parameter
and let $\kappa : (0,1]^3 \rightarrow (0,1]$ be a decay function.
Suppose that $f \in \bump{n}$ is
such that $0 \leq f \leq 1$ and $\Supp f \subset \cube{8}$.
Then there exists a decomposition $f = f_1 + f_2 + f_3$ with
$f_i \in \bump{n}$, $\Supp f_i \subset \cube{4}$, $\|f_i\|_\infty \leq 1$,
$f_1 \geq 0$, $f_1 + f_2 \geq 0$, $\int f_1 = \int f$
as well as a Bohr set $B$ of 
dimension $d \lesssim_{\eps,\kappa} 1$
and radius $\delta \gtrsim_{\eps,\kappa} 1$
such that
\begin{align}
\label{eq:reg:RegLemmaEqs}
	\| T^t f_1 - f_1 \|_\infty 
	\leq \eps \quad\ \forall t \in B,
	\qquad
	\| f_2 \|_2 \leq \eps,
	\qquad
	\| \wh{f}_3 \|_{L^\infty(\R^n)} 
	\leq \kappa(\eps,d^{-1},\delta).
\end{align}
\end{proposition}

\begin{proof}
We initially consider $f$ as defined on the torus $\T^n$, 
by identification with its $1$-periodization from the cube $\cube{2}$.
Consider sequences of positive real numbers
\begin{align*}
	\tfrac{1}{2} \geq \delta_0 \geq \delta_1 \geq \dots \geq \delta_i \geq \dots
	\quad\text{and}\quad
	1 \geq \eta_1 \geq \dots \geq \eta_i \geq \dots
\end{align*}
to be determined later.
We define sequences of frequency sets $\Gamma_i$
and Bohr sets $B_i$ of dimension $d_i$,
and measures $\nu_i$ inductively for $i \geq 0$ by
\begin{align}
	\label{eq:reg:SpectraDef}
	\Gamma_{i+1}
	&= \Gamma_i \bigcup \{\, |\wh{f}| \geq \eta_{i+1} \,\} \bigcup \bigcup_{j=0}^i \{\, |\wh{\nu}_j| \geq \eta_{i+1} \,\},
	\\
	\notag
	B_{i+1} &= B\big( \Gamma_{i+1},\delta_{i+1} \big),
	\qquad \nu_{i+1} = \nu_{B_{i+1}}.
\end{align}
We initialize with $\Gamma_0 = \{ e_1,\dots,e_n \}$, $\delta_0 \leq \tfrac{1}{8}$,
$B_0 = B(\Gamma_0,\delta_0)$, $\nu_0 = \nu_{B_0}$,
so that $d_0 = n$ and by the definition~\eqref{eq:abs:BohrSetDef} of Bohr sets,
we have $B_i \subset \cube{8}$ for all $i$.
Note that, by Tchebychev, we also have a dimension bound
\begin{align*}
	d_{i+1} \leq d_i + \frac{\|\wh{f} \|_2^2}{\eta_{i+1}^2} + \sum_{j = 0}^i \frac{\| \wh{\nu}_j \|_2^2}{\eta_{i+1}^2}.
\end{align*}
By Plancherel and the bound~\eqref{eq:reg:BohrLinftyBound}, 
it follows that
\begin{align}
\label{eq:reg:DimBound}
	&\phantom{(i \geq 1)} &
	d_i &\lesssim_{\delta_0,\dots,\delta_{i-1},d_{i-1},\eta_i} 1
	&&(i \geq 1).
\end{align}

We start by finding a piece of the Fourier expansion
of $f$ which is small in $L^2$.
To this end observe that
\begin{align*}
	\sum_{i=0}^k \sum_{\Gamma_{i+2} \smallsetminus \Gamma_i} |\wh{f}|^2 
	\leq 2\| \wh{f} \|_2^2 = 2\| f \|_2^2 \leq 2.
\end{align*}
By Tchebychev's bound, it follows that
\begin{align*}
	\#\Big\lbrace  0 \leq i \leq k \,:\, 
	\sum_{\Gamma_{i+2} \smallsetminus \Gamma_i} |\wh{f}|^2 \geq \frac{\eps^2}{2} \Big\rbrace  
	\leq \frac{4}{\eps^2}.
\end{align*}
Choosing $k = \lceil 4/\eps^2\rceil$, 
we obtain the existence of an index 
$0 \leq i \leq k$ such that
\begin{align}
\label{eq:reg:SmallL2Piece}
	 \sum_{\Gamma_{i+2} \smallsetminus \Gamma_i} |\wh{f}|^2 \leq \frac{\eps^2}{2}.
\end{align}

We now decompose $f$ into three pieces $f_1,f_2,f_3 : \T^n \rightarrow \C$ defined by
\begin{align*}
	f = f \ast \nu_i
	+ ( f \ast \nu_{i+1} - f \ast \nu_i )
	+ ( f - f \ast \nu_{i+1} )
	= f_1 + f_2 + f_3 .
\end{align*}
Since $f$ takes values in $[0,1]$
and $\int \nu_i = 1$,
the functions $f_1,f_2,f_3$ take values in $[-1,1]$
by simple convolution bounds.
It is also clear that $f_1$ and $f_1 + f_2$ are non-negative
and $\int f_1 = \int f$.

Let us first analyze the $L^2$-small piece.
By Plancherel and~\eqref{eq:reg:SmallL2Piece}, we have
\begin{align}
	\notag
	\| f \ast \nu_{i+1} - f \ast \nu_i \|_2^2
	&= \sum_{m \in \Z^n} |\wh{f}(m)|^2 |\wh{\nu_{i+1}}(m) - \wh{\nu_i}(m)|^2
	\\
	\label{eq:reg:L2Piece}
	&\leq\ \frac{\eps^2}{2}\
	+ \sum_{ m \in \Gamma_i \cup (\Z^n \smallsetminus \Gamma_{i+2}) }
		|\wh{f}(m)|^2 |\wh{\nu_{i+1}}(m) - \wh{\nu_i}(m)|^2	
\end{align}
For $m \in \Gamma_i \subset \Gamma_{i+1}$,
by~\eqref{eq:reg:BohrAnnihilation} we have
\begin{align*}
	|\wh{\nu_{i+1}}(m) - \wh{\nu_i}(m)| \lesssim \delta_{i+1} + \delta_i.
\end{align*}
For $m \not\in \Gamma_{i+2}$
the definition~\eqref{eq:reg:SpectraDef} of $\Gamma_{i+2}$
implies that $|\wh{\nu_i}(m)| \leq \eta_{i+2}$ and $|\wh{\nu_{i+1}}(m)| \leq \eta_{i+2}$.
Inserting these bounds into~\eqref{eq:reg:L2Piece}, we obtain
\begin{align}
\label{eq:reg:SmallL2}
	\| f \ast \nu_{i+1} - f \ast \nu_i \|_2^2
	\leq \frac{\eps^2}{2} + O( \delta_i + \delta_{i+1} + \eta_{i+2} )
	\leq \eps^2,
\end{align}
provided that $\delta_j,\eta_j \leq c\eps^2$ for all $j$.

Next, let us focus on the almost-periodic piece.
Introducing a parameter $\rho_i \in (0,1]$,
we deduce from~\eqref{eq:reg:BohrAlmostPeriodicity} 
that for $t \in B_{\rho_i}$, we have
\begin{align}
	\notag
	\| T^t f \ast \nu_i - f \ast \nu_i \|_\infty
	&\leq \| f \|_1 \| T^t \nu_i - \nu_i \|_\infty 
	\\
	\notag
	&\lesssim_n \delta_i^{-d_i} \rho_i 
	\\
	\label{eq:reg:SmallAlmostPeriodicity}
	&\leq \eps,
\end{align} 
choosing $\rho_i = c_n \eps \delta_i^{d_i}$.
We write $\wt{\delta}_i = \rho_i \delta_i$,
and from~\eqref{eq:reg:DimBound} we see that 
$\wt{\delta_i}$ depends at most on 
$n,\eps,\delta_0,\dots,\delta_i,\eta_1,\dots,\eta_i$.

Finally, we consider the Fourier-small piece.
By Fourier inversion,
\begin{align*}
	\| ( f - f \ast \nu_{i+1} )^\wedge \|_{\ell^\infty(\Z^n)}
	= \sup\limits_{m \in \Z^n} |\wh{f}(m)| | 1 - \wh{\nu}_{i+1}(m) |.
\end{align*}
For $m \in \Gamma_{i+1}$,
we have $|1 - \wh{\nu}_{i+1} (m)| \lesssim \delta_{i+1}$ by~\eqref{eq:reg:BohrAnnihilation},
while for $m \not\in \Gamma_{i+1}$,
the definition~\eqref{eq:reg:SpectraDef} of $\Gamma_{i+1}$ shows that
$|\wh{f}(m)| \leq \eta_{i+1}$.
Therefore
\begin{align}
\label{eq:reg:SmallFourierInfty}
	\| ( f - f \ast \nu_{i+1} )^\wedge \|_{\ell^\infty(\Z^n)}
	\lesssim \delta_{i+1} + \eta_{i+1}
	\leq c \kappa(\eps,d_i^{-1},\wt{\delta}_i),
\end{align}
for a small constant $c > 0$
provided that we choose the $\delta_j,\eta_j$
recursively satisfying
\begin{align*}
	\max(\delta_{i+1},\eta_{i+1}) = c \min( \kappa(\eps,d_i^{-1},\wt{\delta}_i), \eps^2 ).
\end{align*}
At this stage we have obtained
the desired bounds~\eqref{eq:reg:RegLemmaEqs} 
over $\T^n$ and for a Bohr set
$\wt{B}_i = B_i(\Gamma_i,\wt{\delta}_i)$,
and from~\eqref{eq:reg:DimBound} and 
the construction of the $\delta_i$
it follows that
\begin{align*}
	d_i \lesssim_{\eps,\kappa} 1
	\quad\text{and}\quad
	\delta_i \gtrsim_{\eps,\kappa} 1.
\end{align*}

To finish the proof we now consider
the functions $f_1,f_2,f_3$ as functions
on $\R^n$ supported on $\cube{2}$.
Since $f$ and the Bohr sets measures $\nu_i$
are supported on $\cube{8}$,
the convolutions $f \ast \nu_i$ over $\T^n$ may be readily interpreted
as convolutions over $\R^n$, and the functions $f_i$ are supported on $\cube{4}$.
The properties~\eqref{eq:reg:SmallL2} and~\eqref{eq:reg:SmallAlmostPeriodicity}
are readily viewed as holding over $\R^n$, thus we
only need to verify that $f_3$ has the appropriate Fourier decay
at real frequencies.
We claim that since $f_3$ has support in $\cube{4}$,
we have $\| \wh{f}_3 \|_{L^\infty(\R^n)} \lesssim \| \wh{f}_3 \|_{\ell^\infty(\Z^n)} $
and by taking the constant $c$ in~\eqref{eq:reg:SmallFourierInfty}
small enough, we obtain the desired Fourier decay estimate.
To prove this claim, consider a smooth bump function
$\chi$ equal to $1$ on $\cube{4}$.
For $\xi \in \R^n$, expanding $f$ as a Fourier series yields
\begin{align*}
	\wh{f}_3(\xi)
	&= \int_{[-\frac{1}{4},\frac{1}{4}]^n} f_3(x) \chi(x) e( - \xi \cdot x ) \dx
	\\
	&= \sum_{k \in \Z^n} \wh{f}_3(k) \int_{\R^n} \chi(x) e( (k-\xi) \cdot x ) \dx
	\\
	&= \sum_{k \in \Z^n} \wh{f}_3(k) \wh{\chi}( \xi - k).
\end{align*}
Using the smoothness of $\chi$, it follows that, uniformly in $\xi \in \R^n$,
\begin{align*}
	|\wh{f}_3(\xi)|
	\lesssim \| \wh{f}_3 \|_{\ell^\infty(\Z^n)}
	\sum_{k \in \Z^n}
	(1 + |\xi - k|)^{-(n+1)}
	\lesssim \| \wh{f}_3 \|_{\ell^\infty(\Z^n)}.
\end{align*}
\end{proof}

\section{Uniform restriction estimates for fractal measures}
\label{sec:restr}

In this section we obtain restriction estimates
for fractal measures satisfying dimensionality
and Fourier decay conditions, with uniformity
in all the parameters involved.
Throughout this section,
we liberate $\mu, \alpha,\beta$ from their usual meaning,
and we track dependencies on all parameters such as the dimension $n$.
To facilitate our quoting of the literature, 
we first recall the functional equivalences
in Tomas' $T^*T$ argument~\cite[Chapter~7]{Wolff:Book}.

\begin{fact}
\label{thm:restr:RestrEquivs}
Suppose that $\mu \in \meas{n}$ and $p \in (1,+\infty]$,
and that $p'$ is given by $\frac{1}{p} + \frac{1}{p'} = 1$.
Let $R > 0$.
The following statements are equivalent:
\begin{align}
	\label{eq:restr:Restriction}
	&\phantom{\forall f \in \schw{n}} &
	\| \wh{f} \|_{L^2(\dmu)} &\leq R \| f \|_{L^{p'}(\R^n)}
	& \forall\, f &\in \schw{n}, \\
	\label{eq:restr:Extension}
	&\phantom{\forall f \in \schw{n}} &
	\| \wh{g\dmu} \|_{L^p(\R^n)} &\leq R \| g \|_{L^2(\dmu)}
	& \forall\, g &\in L^2(\dmu).
\end{align}
\end{fact}

We now fix two exponents $0 < \beta \leq \alpha \leq n$
and two constants $A,B \geq 1$,
an we restrict our attention to probability measures $\mu$ on $\R^n$ satisfying
\begin{align}
	\label{eq:restr:BallDecay}
	&\phantom{(x \in \R^n, r > 0 )} &
	\mu\big[B(x,r)\big] &\leq A r^{\alpha} &
	&(x \in \R^n, r > 0), \\
	\label{eq:restr:FourierDecay}
	&\phantom{(\xi \in \R^n)} &	
	|\wh{\mu}(\xi)| &\leq B (1 + |\xi|)^{-\beta/2}
	&&(\xi \in \R^n).
\end{align}
We define the critical exponent
\begin{align}
\label{eq:restr:CritExp}
	p_0 = 2 + \frac{4(n-\alpha)}{\beta},
\end{align}
so that the Mitsis-Mockenhaupt restriction 
theorem~\cite{Mitsis:RestrFrac,Mockenhaupt:RestrFrac}
states that each of the inequalities in Fact~\ref{thm:restr:RestrEquivs}
holds for $p > p_0$, for a certain constant $R = R(A,B,\alpha,\beta,p,n)$.
We wish to use~\eqref{eq:restr:Extension} with $g \equiv 1$
and $p = 2 + \delta$ with a fixed small $\delta > 0$, 
which is possible when $\alpha$ is close enough to $n$ by~\eqref{eq:restr:CritExp},
but to be useful this requires some uniformity in $\alpha$.
The constants in~\cite{Mitsis:RestrFrac,Mockenhaupt:RestrFrac}
can be given explicit expressions in terms of the parameters involved,
and in fact one could likely adapt the version of
Mockenhaupt's argument in~\cite[Proposition~4.1]{LP:Configs},
to relax the condition $\beta  > 2/3$ there to $\beta > 0$.
We provide instead a direct derivation from
the estimate of Bak and Seeger~\cite{BK:RestrFrac},
which includes explicit constants.

\begin{proposition}
\label{thm:restr:BKBound}
Let $\beta_0 \in (0,n)$.
There exists $C_{n,\beta_0} > 0$ such that,
when $\beta \geq \beta_0$, 
the estimate~\eqref{eq:restr:Restriction}
holds for $p \geq p_0$ with
$R = C_{n,\beta_0} \max(A,B)^{p_0/2p}$.
\end{proposition}

\begin{proof}
Apply~\cite[Eq.~(1.5)]{BK:RestrFrac},
replacing $a \leftarrow \alpha$, $b \leftarrow \beta/2$,
$d \leftarrow n$, $p \leftarrow p'$,
so that $q = \frac{2p}{p_0}$;
and note that $\alpha,\beta$ belong to the compact interval $[\beta_0,n]$.
Since $q \geq 2$ for $p \geq p_0$,
by nesting of $L^s(\dmu)$ norms this yields
\begin{align*}
	\| \wh{f} \|_{L^{2}(\dmu)} 
	&\leq
	\| \wh{f} \|_{L^{q}(\dmu)} \\
	&\leq (C_{n,\beta_0})^{\tfrac{2}{q}}
	A^{\tfrac{1}{q} \cdot \tfrac{2}{p_0}} B^{ \frac{1}{q} \big(1 - \tfrac{2}{p_0}\big) } \| f \|_{L^{p'}(\R^n)} \\
	&\leq C_{n,\beta_0} \max(A,B)^{\tfrac{p_0}{2p}} \| f \|_{L^{p'}(\R^n)}.
\end{align*}
\end{proof}

Alternatively, one may choose to
track down the dependencies on constants
in Mitsis' simpler argument~\cite{Mitsis:RestrFrac},
which would lead to a similar estimate 
for the constant $R$ in~\eqref{eq:restr:Restriction}, 
upto a harmless (for our argument) factor $(p-p_0)^{-1}$.
Via Proposition~\ref{thm:restr:BKBound},
it is now possible to bound the moments of $\wh{\mu}$
of order slightly larger than $2$
when $\alpha$ is close enough to $n$, 
with only a moderate dependency of constants on $\alpha$.

\begin{proposition}
\label{thm:restr:MuMoment}
Let $\delta \in (0,1)$ and $\beta_0 \in (0,n)$.
Suppose that $\mu$ is a probability measure 
satisfying~\eqref{eq:scheme:BallDecay} and~\eqref{eq:scheme:FourierDecay}.
Then, uniformly for $n - \tfrac{\delta\beta_0}{4} \leq \alpha < n$
and $\beta_0 \leq \beta < n$, we have
\begin{align*}
	\| \wh{\mu} \|_{2 + \delta}
	\lesssim_{\beta_0,n} D_\alpha^{1/2}.
\end{align*}
\end{proposition}

\begin{proof}
We consider the exponent $p = 2 + \delta$.
Recalling~\eqref{eq:restr:CritExp}, 
we have $p \geq p_0$ in the stated range of $\alpha$.
We can therefore invoke Proposition~\ref{thm:restr:BKBound}
with $A = D \asymp 1$ and $B = D_\alpha$,
so that the extension inequality~\eqref{eq:restr:Extension} 
holds for $g \equiv 1$ with 
$R \lesssim_{\beta_0,n} D_\alpha^{1/2}$.
\end{proof}

\bibliographystyle{amsplain}
\bibliography{quadfractals_revised_arxiv2}

\providecommand{\bysame}{\leavevmode\hbox to3em{\hrulefill}\thinspace}
\providecommand{\MR}{\relax\ifhmode\unskip\space\fi MR }
\providecommand{\MRhref}[2]{%
  \href{http://www.ams.org/mathscinet-getitem?mr=#1}{#2}
}
\providecommand{\href}[2]{#2}
\begin{thebibliography}{10}

\bibitem{BK:RestrFrac}
J.-G. Bak and A.~Seeger, \emph{Extensions of the {S}tein-{T}omas theorem},
  Math. Res. Lett. \textbf{18} (2011), no.~4, 767--781.

\bibitem{Baker:Book}
R.~C. Baker, \emph{Diophantine inequalities}, The Clarendon Press, Oxford
  University Press, New York, 1986.

\bibitem{BIT:Chains}
M.~Bennett, A.~Iosevich, and K.~Taylor, \emph{Finite chains inside thin subsets
  of $\mathbb{R}^d$}, to appear in Anal. PDE,
  \url{http://arxiv.org/abs/1409.2581}.

\bibitem{BL:PolSzemeredi}
V.~Bergelson and A.~Leibman, \emph{Polynomial extensions of van der {W}aerden's
  and {S}zemer\'edi's theorems}, J. Amer. Math. Soc. \textbf{9} (1996), no.~3,
  725--753.

\bibitem{BLL:PolSzemeredi}
V.~Bergelson, A.~Leibman, and E.~Lesigne, \emph{Intersective polynomials and
  the polynomial {S}zemer\'edi theorem}, Adv. Math. \textbf{219} (2008), no.~1,
  369--388.

\bibitem{BM:PolSzemeredi}
V.~Bergelson and R.~McCutcheon, \emph{An ergodic {IP} polynomial {S}zemer\'edi
  theorem}, Mem. Amer. Math. Soc. \textbf{146} (2000), no.~695, viii+106.

\bibitem{Bluhm:SalemI}
C.~Bluhm, \emph{Random recursive construction of {S}alem sets}, Ark. Mat.
  \textbf{34} (1996), no.~1, 51--63.

\bibitem{Bluhm:SalemII}
\bysame, \emph{On a theorem of {K}aufman: {C}antor-type construction of linear
  fractal {S}alem sets}, Ark. Mat. \textbf{36} (1998), no.~2, 307--316.

\bibitem{Bourgain:DilatesConfig}
J.~Bourgain, \emph{A {S}zemer\'edi type theorem for sets of positive density in
  {${\bf R}^k$}}, Israel J. Math. \textbf{54} (1986), no.~3, 307--316.

\bibitem{CLP:Configs}
V.~Chan, I.~{\L}aba, and M.~Pramanik, \emph{Finite configurations in sparse
  sets}, to appear in J. Anal. Math., \url{http://arxiv.org/abs/1307.1174}.

\bibitem{Chen:Salem}
X.~Chen, \emph{Sets of {S}alem type and sharpness of the ${L}^2$ {F}ourier
  restriction theorem}, to appear in Trans. Amer. Math. Soc.,
  \url{http://arxiv.org/abs/1305.5584}.

\bibitem{Erdogan:DistanceSets2}
M.~B. Erdo{\~g}an, \emph{A bilinear {F}ourier extension theorem and
  applications to the distance set problem}, Int. Math. Res. Not. (2005),
  no.~23, 1411--1425.

\bibitem{Erdogan:DistanceSets1}
\bysame, \emph{On {F}alconer's distance set conjecture}, Rev. Mat. Iberoam.
  \textbf{22} (2006), no.~2, 649--662.

\bibitem{Furstenberg:Sarkozy}
H.~Furstenberg, \emph{Ergodic behavior of diagonal measures and a theorem of
  {S}zemer\'edi on arithmetic progressions}, J. Analyse Math. \textbf{31}
  (1977), 204--256.

\bibitem{FK:MultiSzemeredi}
H.~Furstenberg and Y.~Katznelson, \emph{An ergodic {S}zemer\'edi theorem for
  commuting transformations}, J. Analyse Math. \textbf{34} (1978), 275--291
  (1979).

\bibitem{GW:Complexity}
W.~T. Gowers and J.~Wolf, \emph{The true complexity of a system of linear
  equations}, Proc. Lond. Math. Soc. (3) \textbf{100} (2010), no.~1, 155--176.

\bibitem{GGIP:GenConfigs}
L.~Grafakos, A.~Greenleaf, A.~Iosevich, and E.~Palsson, \emph{Multilinear
  generalized {R}adon transforms and point configurations}, Forum Math.
  \textbf{27} (2015), no.~4, 2323--2360.

\bibitem{GT:DiophApprox}
B.~Green and T.~Tao, \emph{New bounds for {S}zemer\'edi's theorem. {II}. {A}
  new bound for {$r_4(N)$}}, Analytic number theory, Cambridge Univ. Press,
  Cambridge, 2009, pp.~180--204.

\bibitem{GI:GenConfigs}
A.~Greenleaf and A.~Iosevich, \emph{On triangles determined by subsets of the
  {E}uclidean plane, the associated bilinear operators and applications to
  discrete geometry}, Anal. PDE \textbf{5} (2012), no.~2, 397--409.

\bibitem{GILP:GenConfigs}
A.~Greenleaf, A.~Iosevich, B.~Liu, and E.~Palsson, \emph{A group-theoretic
  viewpoint on {E}rd{\H{o}}s-{F}alconer problems and the {M}attila integral},
  Preprint (2014), \url{http://arxiv.org/abs/1306.3598}.

\bibitem{GIP:Necklaces}
A.~Greenleaf, A.~Iosevich, and M.~Pramanik, \emph{On necklaces inside thin
  subsets of $\mathbb{R}^d$}, to appear in Math. Res. Lett.,
  \url{http://arxiv.org/abs/1409.2588}.

\bibitem{Hambrooke:Salem}
K.~Hambrooke, \emph{Explicit {S}alem sets and applications to metrical
  {D}iophantine approximation}, Preprint (2014),
  \url{http://www.math.rochester.edu/people/faculty/khambroo/research/index.html}.

\bibitem{IL:Triangles}
A.~Iosevich and B.~Liu, \emph{Equilateral triangles in subsets of
  $\mathbb{R}^d$ of large {H}ausdorff dimension}, Preprint (2016),
  \url{http://arxiv.org/abs/1603.01907}.

\bibitem{Kahane:Book}
J.-P. Kahane, \emph{Some random series of functions}, second ed., Cambridge
  University Press, Cambridge, 1985.

\bibitem{Kaufman:Salem}
R.~Kaufman, \emph{On the theorem of {J}arn\'\i k and {B}esicovitch}, Acta
  Arith. \textbf{39} (1981), no.~3, 265--267.

\bibitem{Keleti:NoConfigs}
T.~Keleti, \emph{A 1-dimensional subset of the reals that intersects each of
  its translates in at most a single point}, Real Anal. Exchange \textbf{24}
  (1998/99), no.~2, 843--844.

\bibitem{Korner:Salem}
T.~W. K{\"o}rner, \emph{Hausdorff and {F}ourier dimension}, Studia Math.
  \textbf{206} (2011), no.~1, 37--50.

\bibitem{LP:Configs}
I.~{\L}aba and M.~Pramanik, \emph{Arithmetic progressions in sets of fractional
  dimension}, Geom. Funct. Anal. \textbf{19} (2009), no.~2, 429--456.

\bibitem{LM:DiophApprox}
N.~Lyall and A.~Magyar, \emph{Simultaneous polynomial recurrence}, Bull. Lond.
  Math. Soc. \textbf{43} (2011), no.~4, 765--785.

\bibitem{Maga:NoConfigs}
P.~Maga, \emph{Full dimensional sets without given patterns}, Real Anal.
  Exchange \textbf{36} (2010/11), no.~1, 79--90.

\bibitem{Mathe:NoConfigs}
A.~Math\'{e}, \emph{Sets of large dimension not containing polynomial
  configurations}, Preprint (2012), \url{http://arxiv.org/abs/1201.0548}.

\bibitem{MS:DistanceSets}
P.~Mattila and P.~Sj{\"o}lin, \emph{Regularity of distance measures and sets},
  Math. Nachr. \textbf{204} (1999), 157--162.

\bibitem{Mitsis:RestrFrac}
T.~Mitsis, \emph{A {S}tein-{T}omas restriction theorem for general measures},
  Publ. Math. Debrecen \textbf{60} (2002), no.~1-2, 89--99.

\bibitem{Mockenhaupt:RestrFrac}
G.~Mockenhaupt, \emph{Salem sets and restriction properties of {F}ourier
  transforms}, Geom. Funct. Anal. \textbf{10} (2000), no.~6, 1579--1587.

\bibitem{Nicolaescu:Coarea}
L.~Nicolaescu, \emph{The coarea formula}, Expository note (2011),
  \url{http://www3.nd.edu/~lnicolae/vita.html}.

\bibitem{Orponen:DistanceSets}
T.~Orponen, \emph{On the distance sets of {AD}-regular sets}, Preprint (2015),
  \url{http://arxiv.org/abs/1509.06675}.

\bibitem{Roth:RothI}
K.~F. Roth, \emph{On certain sets of integers}, J. London Math. Soc.
  \textbf{28} (1953), 104--109.

\bibitem{Roth:RothII}
\bysame, \emph{On certain sets of integers. {II}}, J. London Math. Soc.
  \textbf{29} (1954), 20--26.

\bibitem{Salem:Salem}
R.~Salem, \emph{On singular monotonic functions whose spectrum has a given
  {H}ausdorff dimension}, Ark. Mat. \textbf{1} (1951), 353--365.

\bibitem{Sarkozy:Sarkozy}
A.~S{\'a}rk{\H{o}}zy, \emph{On difference sets of sequences of integers. {I}},
  Acta Math. Acad. Sci. Hungar. \textbf{31} (1978), no.~1--2, 125--149.

\bibitem{Shmerkin:Configs}
P.~Shmerkin, \emph{Salem sets with no arithmetic progressions}, Preprint,
  \url{http://arxiv.org/abs/1510.07596}.

\bibitem{Stein:Book}
E.~M. Stein, \emph{Harmonic analysis: real-variable methods, orthogonality, and
  oscillatory integrals}, Princeton University Press, Princeton, NJ, 1993.

\bibitem{Tao:U2RegLemma}
T.~Tao, \emph{A proof of {R}oth's theorem}, Blog post (2014),
  \url{https://terrytao.wordpress.com/2014/04/24/a-proof-of-roths-theorem/}.

\bibitem{TV:Book}
T.~Tao and V.~Vu, \emph{Additive combinatorics}, Cambridge University Press,
  Cambridge, 2006.

\bibitem{Wolff:DistanceSets}
T.~Wolff, \emph{Decay of circular means of {F}ourier transforms of measures},
  Internat. Math. Res. Notices (1999), no.~10, 547--567.

\bibitem{Wolff:Book}
T.~H. Wolff, \emph{Lectures on harmonic analysis}, American Mathematical
  Society, Providence, RI, 2003.

\end{thebibliography}

\bigskip

\textsc{\small Kevin Henriot}

\textsc{\footnotesize Department of mathematics,
University of British Columbia,
Room 121, 1984 Mathematics Road,
Vancouver BC V6T 1Z2, Canada
}

\textit{\small Email address: }\texttt{\small khenriot@math.ubc.ca}

\medskip

\textsc{\small Izabella \L{}aba}

\textsc{\footnotesize Department of mathematics,
University of British Columbia,
Room 121, 1984 Mathematics Road,
Vancouver BC V6T 1Z2, Canada
}

\textit{\small Email address: }\texttt{\small ilaba@math.ubc.ca}

\medskip

\textsc{\small Malabika Pramanik}

\textsc{\footnotesize Department of mathematics,
University of British Columbia,
Room 121, 1984 Mathematics Road,
Vancouver BC V6T 1Z2, Canada
}

\textit{\small Email address: }\texttt{\small malabika@math.ubc.ca}

\end{document}